\documentclass[11pt,reqno]{amsart} 
\usepackage{booktabs}
\usepackage[margin=0.8in]{geometry}
\usepackage{amsmath, amsfonts, amssymb,  mathrsfs,  array,  amsthm, bm,bbm,mathtools}
\usepackage{graphicx, float, enumitem, tabularx, color}
\usepackage[colorlinks=true, linkcolor=blue, citecolor=blue, urlcolor=blue]{hyperref}
\usepackage{setspace}   
\usepackage{blkarray}

\usepackage{microtype}   
\newtheorem{definition}{Definition}[section]
\newtheorem{theorem}{Theorem}[section]
\newtheorem{lemma}[theorem]{Lemma}
\newtheorem{proposition}[theorem]{Proposition}
\newtheorem{corollary}[theorem]{Corollary}
\newtheorem{remark}[theorem]{Remark}
\newtheorem{assumption}[theorem]{Assumption}

\newcommand{\bR}{\mathbb{R}}
\newcommand{\bP}{\mathbb{P}}
\newcommand{\bE}{\mathbb{E}}
\newcommand{\bF}{\mathbb{F}}
\newcommand{\bX}{\mathbb{X}}
\newcommand{\bN}{\mathbb{N}}
\newcommand{\bT}{\mathbb{T}}
\newcommand{\one}{\mathbbm{1}}
\newcommand{\X}{\mathbb{X}}

\newcommand{\cA}{\mathcal{A}}
\newcommand{\cB}{\mathcal{B}}

\newcommand{\cF}{\mathcal{F}}

\newcommand{\cP}{\mathcal{P}}
\newcommand{\cR}{\mathcal{R}}
\newcommand{\cQ}{\mathcal{Q}}
\newcommand{\cL}{\mathcal{L}}
\newcommand{\cH}{\mathcal{H}}
\title{Stackelberg stopping games}
\author[]{Jingjie Zhang}
\address{China School of Banking and Finance, University of International Business and Economics}
\email{jingjie.zhang@uibe.edu.cn}
\author[]{Zhou Zhou}
\address{School of Mathematics and Statistics, University of Sydney}
\email{zhou.zhou@sydney.edu.au}
\date{\today}

\begin{document}

\begin{abstract}
We study a Stackelberg variant of the classical discrete-time Dynkin game, in which Player 1 (the leader) commits to a stopping strategy first and Player 2 (the follower) responds optimally. This leader–follower structure induces an optimal control problem for the leader and gives rise to intrinsic time-inconsistency.

We first clarify notions of precommitment and equilibrium strategies in the Stackelberg setting, and contrast them with the Nash equilibrium in the standard Dynkin game using a finite-horizon example. We then consider an infinite-horizon framework with a time-homogeneous Markov process on a general Polish state space. We characterize the leader's value function under randomized precommitment strategies and show that randomized exact equilibrium strategies may fail to exist via a counterexample. Motivated by this nonexistence phenomenon, we introduce an entropy-regularized Stackelberg stopping game. The regularization induces a continuous response rule and yields the existence of randomized regular equilibria. We further show that these regular equilibria induce $\varepsilon$-equilibria for the original Stackelberg stopping game when the regularization parameter is sufficiently small. In the finite-state setting, we also establish a limiting result as the regularization parameter converges to zero.
\end{abstract}

\maketitle

\section{Introduction}

In a classical Dynkin game, two players interact through their stopping strategies. The game ends when one of the players stops, and the payoffs depend on which player stops first. Since its introduction in \cite{Dynkin69}, the Dynkin game has been studied in various directions. Works on zero-sum Dynkin games include \cite{Dynkin69, Neveu-book-75} in discrete time, \cite{Bismut77, LM84, Morimoto84} in continuous time, \cite{Yasuda85, RSV01, TV02}  with randomized strategies, \cite{CK96} within non-Markovian settings, and \cite{BY17} under model uncertainty. The nonzero-sum version of the game has been studied in \cite{Morimoto86, Ohtsubo87, Nagai87} and more recently in \cite{HZ09, LS13, DFM18}, among others. It is worth noting that all the works mentioned above focus on Nash equilibrium in the sense that in the stopping games the two players choose their stopping strategies simultaneously. In many practical situations, however, strategic interactions are inherently asymmetric: one party may commit to a policy in advance, while the other responds optimally after observing this commitment, as in regulatory interventions, contract design, or oligopoly market with a dominant firm.  Motivated by this, we study a variant of the Dynkin game in which the two players determine their stopping strategies {\it sequentially}, giving rise to a Stackelberg (leader–follower) stopping game.
This framework provides a natural extension of standard Dynkin game problems, such as game options or convertible bonds, to settings with asymmetric timing.

Time inconsistency in Stackelberg games was first observed in \cite{MR475954}, where the author studies a dynamic game with one dominant player and multiple nondominant players. Under open-loop and closed-loop policies, the original plan may no longer be optimal as time progresses, leading to a temptation for the dominant player to deviate from the initial plan. As the author notes, ``only the feedback solution has the characteristic that the original plan will not be changed under replanning'' \cite{MR475954}. As we point out in Section~\ref{sec:concepts}, such feedback solution corresponds to our equilibrium strategy defined in Definition~\ref{def:consistent_equilibrium}. 
More recently, \cite{MR4781585} provides an alternative perspective on time inconsistency in Stackelberg games by reformulating the leader's problem as an optimal control problem for forward--backward stochastic differential equations.
The key feature is that the follower's optimal response, which depends on the entire strategy announced by the leader, enters the leader's problem endogenously. As a result, the leader's continuation optimization problem at the next step generally differs from the problem evaluated at the initial time. This leads to intrinsic time inconsistency even in the absence of non-exponential discounting or other exogenous sources of inconsistency. A detailed analysis of this phenomenon in a concrete Stackelberg game model is given in \cite{MR4557638}. In the present paper, the same mechanism appears in the Stackelberg stopping framework; see Proposition~\ref{prop:time_inconsistency} and Remark \ref{rmk:explain_inconsistency} for explicit illustrations.

In the literature of time inconsistency, the seminal work \cite{Strotz55} proposes three approaches to deal with time inconsistency. The first approach is to consider a precommitment strategy, a strategy that is optimal with respect to (w.r.t.)~the initial preference. 
The second approach is to use a naive strategy, under which the agent keeps solving the new problem based on her current preference and thus keeps updating her strategy. The third approach is to look for a time-consistent equilibrium strategy: given that all the future selves follow this strategy, the current self has no incentive to deviate. Recently there has been a lot of research on equilibrium strategies for time-inconsistent problems. See e.g., \cite{EL06,EP08, EMP12, BMZ14,MR4387700,HJZ12} and references therein for time-inconsistent control, and  \cite{CL18, CL20, HN18, HZ19,https://doi.org/10.1111/mafi.12293,MR4578541}, among others, for time-inconsistent stopping. Time-consistent equilibrium strategies for Stackelberg games have received relatively less attention. We refer to \cite{MR475954,MR4557638,MR4781585,MR3979301,MR4348773} and the included references for the related literature.

In this paper, we consider a Stackelberg version of the two-player nonzero-sum Dynkin game in discrete time. In the game, Player 1 (the leader, she) announces her (randomized) stopping strategy first, and then Player 2 (the follower, he) chooses his stopping strategy correspondingly. We do not impose any restriction on the payoff processes for both players and the game is not necessarily of war-of-attrition type. The underlying process is modeled by a time-homogeneous Markov process. We do not include any exogenous  source of time inconsistency (such as non-exponential discounting or non-linear functional of expectation) that commonly appears in a single-agent time-inconsistent problem. The absence of exogenous time inconsistency highlights that the time inconsistency in our framework arises intrinsically from the leader--follower structure of the stopping game.

We start with the finite-horizon setup for the purpose of motivation and concepts. We first consider pure stopping strategies and show in an deterministic example (Proposition \ref{prop:time_inconsistency}) that the leader's problem is time inconsistent. We then define and compare different notions of solutions, including leader's (pure) precommitment and time-consistent equilibrium strategies for the Stackelberg game, as well as the Nash equilibrium for the classical Dynkin game. We demonstrate via an example (Proposition \ref{prop:distinct_notions}) that these notions of solutions are distinct from each other, by showing that leader's initial strategies are different under different solution concepts. Furthermore, three classical solution concepts in dynamic Stackelberg games are contrasted with the precommitment and equilibrium strategies. Their connections and distinctions are emphasized in Remark~\ref{rmk:distinction}. Next, we consider randomized precommitment and time-consistent equilibrium strategies. We observe that randomized strategies may yield strictly larger payoffs for the leader than any pure strategy does (Remark \ref{r1}). We also present a deterministic two-period example (Proposition \ref{prop:precomit_random_nonexist}) indicating that an optimal precommitment strategy may not exist. The nonexistence of an optimal strategy is due to the discontinuity of the follower's best response w.r.t.~the leader's strategy. 

Motivated by the examples and results in the finite-horizon setup, we next consider an infinite-horizon time-homogeneous framework in which the payoff processes are exponentially discounted and the underlying Markov process takes values in a general Polish space.
We first investigate the leader's optimal value induced by randomized path-dependent (precommitment) strategies. The problem is formally defined in \eqref{eq:game_precomit} and the leader's problem is given by \eqref{eq:leader_problem_infT}. Since the leader's problem is time inconsistent, the dynamic programming principle (DPP) cannot be directly applied. Inspired by \cite{MR4557638,MR909450}, we view the follower's continuation utility as the state variable of the leader's value function. Combined with a careful analysis of the admissible range of the follower's utility, this allows us to recover a DPP for the leader's value function (Theorem~\ref{thm:v_bellman}).

We then study randomized equilibrium strategies for the Stackelberg stopping game where the players adopt randomized Markov stopping policies. The problem is formulated in \eqref{eq:game_exact_equilibrium}, for which we introduce two versions of equilibrium, the feedback equilibrium (Definition \ref{def:feedback_equilibrium_general}) and exact equilibrium (Definition \ref{def:infinite_equilibrium}). Their connection and distinction are discussed in Remark \ref{rmk:feedback_vs_exact}. Roughly speaking, the exact equilibrium is a more restricted version of feedback equilibrium that requires that the follower always chooses the earliest optimal stopping time. We construct an example (Proposition~\ref{prop:ex_inf_horz}) showing that exact equilibria may fail to exist even in the presence of randomized strategies. The nonexistence is caused by the discontinuity of the follower's best-response map with respect to the leader's strategy, since the follower's optimal response is binary and non-randomized. In the finite-state setting, we establish the existence of feedback equilibria  (Proposition \ref{prop:feedback}), but it remains an open question for the general Polish space.

Motivated by this difficulty, we introduce an entropy-regularized Stackelberg stopping game in which the follower's payoff is perturbed by an entropy term (see problem \eqref{eq:game_regular_equilibrium}). The regularization induces a continuous randomized response rule of logit type and yields a tractable regularized equilibrium framework. Entropy regularization and logit-type responses are widely used in optimization, learning, and quantal-response models, where they can be interpreted in terms of imperfect optimization, exploration, or information-processing effects (see, for example \cite{MCKELVEY19956} and \cite{JMLR:v21:19-144}). In the present paper, the regularization primarily serves to smooth the follower's best-response map and to obtain a well-posed equilibrium notion, regular equilibrium (Definition \ref{def:regular_equilibrium}), in the general Polish-space setting. The notions of exact equilibrium, feedback equilibrium, regular equilibrium, and $\varepsilon$-equilibrium correspond respectively to equilibrium concepts for the original Stackelberg game, its relaxed feedback formulation, the entropy-regularized game, and approximate equilibria for the original game.

Under the entropy regularization, we establish the existence of regular equilibrium strategies (Theorem~\ref{thm}). We further prove that, when the regularization parameter is sufficiently small, the resulting regularized equilibria induce $\varepsilon$-equilibria (Definition \ref{def:epsilon_equilibrium}) for the original Stackelberg stopping game (Proposition \ref{prop:approximation}). In the special case where the Markov process has a finite state space, we show that as the entropy parameter converges to zero, regular equilibria admit limit points corresponding to feedback equilibria for the original game (Proposition~\ref{prop:feedback}). However, the compactness and continuity arguments used in the finite-state setting do not extend to general Polish spaces, and the existence of (exact) feedback equilibria in the general Polish-space framework remains open (see Remark \ref{rmk:uncountable}). An application in the game option is provided in Section \ref{sec:application}.

Our paper makes a very novel and conceptual contribution to the research of stopping games. Previous works on Dynkin games focus on the existence and construction of Nash equilibrium under which players make decisions simultaneously. To the best of our knowledge, this is the first work that investigates the Stackelberg version of Dynkin games. Our paper introduces a general framework for Stackelberg stopping games in discrete time. We show that the leader--follower structure generates intrinsic time inconsistency and leads to new phenomena that do not arise in classical Dynkin games with simultaneous actions. In particular, we characterize the leader's value function under precommitment strategies, identify nonexistence phenomena for exact randomized equilibria, establish existence results for entropy-regularized equilibria in general Polish spaces, and derive approximation connections between regularized equilibria and the original stopping game.

Our paper also contributes to the literature of time-inconsistent problems. Most of the works on equilibrium strategies for time-inconsistent problems consider single-agent cases, and results for multi-agent settings are relatively limited. The investigation of equilibria for time-inconsistent games is conceptually richer, since players not only play against other agents but also their future selves, and thus the problem can be viewed as a two-layer game. Let us mention the works \cite{MR4397932,MR4783065} on time-inconsistent stopping games. \cite{MR4397932} studies a Dynkin game and \cite{MR4783065} analyzes a mean-field stopping game. In \cite{MR4397932,MR4783065} players take action simultaneously, the time inconsistency originates from non-exponential discounting, and time-consistent equilibria are investigated. In contrast, in our framework the time inconsistency comes from the leader-follower game structure, and both precommitment and equilibrium strategies are considered. The model and methods involved in our paper are also very different from those in \cite{MR4397932,MR4783065}.

The rest of the paper is organized as follows. Section 2 provides motivation and introduces the main concepts via a finite-horizon setting. We illustrate the time inconsistency of the Stackelberg stopping game with a deterministic example and introduce the notions of precommitment and equilibrium strategies, comparing them with the Nash equilibrium in the classical Dynkin game (Section~\ref{sec:finte_pure}). Further comparison with classical solution notions is discussed in Section~\ref{sec:concepts}. We then extend the discussion to randomized strategies (Section~\ref{sec:finite_random}). Section 3 focuses on randomized strategies in the infinite-horizon setup. In Section~\ref{sec:infinity_precommit}, we characterize the leader's value induced by randomized precommitment (path-dependent) strategies and show that the optimal value may not be attained due to the discontinuity of the follower’s response. 
Section~\ref{sec:infinity_equilibrium} discusses equilibrium strategies for the Stackelberg stopping game with randomized Markov policies. The definitions of feedback equilibrium, exact equilibrium and $\varepsilon$-equilibrium are given in this section,  and a counterexample is presented to show that exact randomized equilibria may fail to exist.   
Section~\ref{sec:regular} studies the entropy-regularized Stackelberg stopping game. We define regular equilibrium, establish its existence in the general Polish-space setting, and prove that regular equilibria induce $\varepsilon$-equilibria for the original Stackelberg stopping game. The finite-state case and an application to game options are discussed in Section~\ref{sec:special_case}. All proofs for Section~\ref{sec:infinite} are collected in Section~\ref{sec:proof}.

\section{The Finite-Horizon Stackelberg Stopping Game} \label{sec:finite}

\subsection{Pure Strategies and Motivation}\label{sec:finte_pure}
Denote $\bN:=\{1,2,\dotso\}$, $\bN_0:=\{0,1,2,\dotso\}$ and $[[s,t]]:=\{s,s+1,\dotso,t\}$ for $s,t\in\bN_0$ with $s\leq t$. Take $N,T\in\bN$. Let $X = (X_t)_{t\in \bN_0}$ be a time-homogeneous discrete-time Markov chain with the state space $\bX := \{1,2,\cdots,N\}$. The transition matrix of $X$ is denoted by $\pi = (\pi_{xy})_{x,y \in \bX}$. For $t \in [[0,T]]$ and $x\in \bX$, let $X^{t,x} = (X^{t,x}_s)_{s\in [[t,T]]}$ be the Markov chain starting from time $t$ with the initial state $X_t = x$.
Let $\mathbb{P}_{t,x}$ and $\mathbb{E}_{t,x}$ be the probability measure and expectation conditioned on $X_t = x$. Let $\bF^{X} := \{\cF^{X}_t\}_{t\in [[0,T]]}$ be the filtration generated by $X$. For each $t \in [[0,T]]$, we denote by $\bT_t$ the set of all $\bF^{X}$-stopping times $\tau$ with $t\leq\tau\leq T$. For $t = 0$, we simply write $\bT_0$ as $\bT$. Denote the path space of the Markov chain $X$ over time $[[t, T]]$ by 
\begin{equation*}
\Omega_{t} := \left\{\omega = (\omega_t, \cdots, \omega_T) \mid \omega_s \in \bX, \forall s = t, \cdots, T\right\}.
\end{equation*} 
For $t = 0$, we write $\Omega_0$ as $\Omega$ for simplicity. Let $f_i,g_i,h_i:[[0,T]]\times\X\mapsto\mathbb{R}$ be payoff functions for $i=1,2$.

Consider a discrete-time, finite-horizon, two-player stopping game in which both players observe the Markov chain $X$ and choose stopping times to maximize their expected payoffs. The stopping game has a finite horizon $T$. Let $\tau$ and $\rho$ be the stopping times chosen by Player 1 (leader, she) and Player 2 (follower, he), respectively. For each $t\in [[0,T]]$ and $x \in \bX$, given a stopping time $\tau \in \bT_t$ selected by Player 1, Player 2 responds with a stopping time $\rho \in \bT_t$ that maximizes his expected payoff:
\begin{equation*}
	J_2(t,x;\tau,\rho) : = \mathbb{E}_{t,x}[F_2(\tau,\rho)],
\end{equation*}
where the payoff function $F_2$ is given by
\begin{equation*}
	F_2(\tau, \rho) = f_2(\rho,X_{\rho})\mathbbm{1}_{\{\rho<\tau\}} + g_2(\tau,X_{\tau})\mathbbm{1}_{\{\tau<\rho\}} +h_2(\rho,X_{\rho})\mathbbm{1}_{\{\rho = \tau\}}.
\end{equation*}
For simplicity of the analysis, we assume that Player 2 always chooses the \emph{earliest} stopping time that achieves the maximum expected payoff. We denote this best response by $\rho^*=\rho^*(\tau)$, which is uniquely given by
\begin{equation*}
\rho^*(\tau) = \rho^*_{t,x}(\tau) := \inf\left\{s \ge t \mid F_2(\tau, s) = \sup_{\rho \in \bT_s} \bE\left[F_2(\tau, \rho)|\mathcal{F}_s^X\right]\right\}.
\end{equation*}
Given Player 2's optimal response $\rho^*$, Player 1 chooses $\tau$ to maximize her own expected payoff:
\begin{equation*}
	J_1(t,x;\tau,\rho^*) : = \mathbb{E}_{t,x}[F_1(\tau,\rho^*)],
\end{equation*}
with payoff function
\begin{equation*}
	F_1(\tau, \rho) = f_1(\tau, X_{\tau})\mathbbm{1}_{\{\tau<\rho\}} + g_1(\rho,  X_{\rho})\mathbbm{1}_{\{\rho<\tau\}} +h_1(\tau, X_{\tau})\mathbbm{1}_{\{\rho = \tau\}}.
\end{equation*}
The Stackelberg stopping game reduces to an optimization problem for the leader:
\begin{equation} \label{eq:leader_problem}
\sup_{\tau \in \bT_t} V_{t,x}(\tau),\quad\text{with}\quad V_{t,x}(\tau) :=  J_1(t,x; \tau, \rho^*(\tau)).
\end{equation}

\begin{remark} This model can be interpreted as a project with  termination and withdrawal rights. The leader is a sponsor and the follower is a developer. Sponsor termination liquidates the project; developer withdrawal also ends it, with a potentially different recovery value for the sponsor. 

The role of commitment in project financing and the timing of funding termination have been studied in richer settings; see, for example, Bergemann and Hege~\cite{bergemann2005financing}. Our paper focuses on the strategic interaction between the two parties' stopping decisions under general nonzero-sum payoffs.

In Proposition \ref{prop:time_inconsistency} and Remark \ref{r4}, we can see how the sponsor's future termination policy affects the developer's current decision to remain in the project or withdraw. An initially optimal commitment policy may induce the
developer to remain, yet the sponsor may prefer to revise that policy once a later decision date is reached. In Proposition \ref{prop:distinct_notions}, we ompare this precommitment
benchmark with an equilibrium policy under which, given the behavior of future selves and the developer's optimal response, the sponsor has no incentive to revise its current decision.
\end{remark}

\begin{remark}
If the two players make their decisions simultaneously, the game is then a classical nonzero-sum Dynkin game, a framework that has been extensively studied. In contrast, we consider a Stackelberg game setting, where Player 1 (the leader) commits to a strategy first, and Player 2 (the follower) responds after observing this strategy. To the best of our knowledge, this is the first study to examine such a Stackelberg (leader-follower) structure in the context of two-player stopping games.
\end{remark}

We define the optimal solution in Problem~\eqref{eq:leader_problem} as follows.

\begin{definition} \label{def:precom_finite_pure}
For each fixed $t\in [[0,T]]$ and $x\in \bX$, a stopping policy $\tau_{t,x}^* \in \bT_t$ is called a pure precommitment stopping strategy (w.r.t.~the initial time $t$) in Problem~\eqref{eq:leader_problem} if $V_{t,x}(\tau_{t,x}^*) = \sup_{\tau \in \bT_t} V_{t,x}(\tau)$.
\end{definition}

The term \emph{precommitment strategy} originates in the literature on time inconsistency. It refers to a strategy that is optimal when chosen at the initial time, but may cease to be optimal from the perspective of future selves. As we see in the following proposition, Problem ~\eqref{eq:leader_problem} exhibits time inconsistency. For each fixed $t\in [[0,T]]$ and $x\in \bX$, the stopping policy $\tau^*_{t,x}$ is optimal in a static sense. That is, the decision is made at time $t$ without taking into account the change of preferences in the future. Therefore, we refer to $\tau^*_{t,x}$ as the precommitment strategy.

\begin{proposition} \label{prop:time_inconsistency}
 Problem~\eqref{eq:leader_problem} may be time inconsistent. That is, $\tau^*_{t,x}(\omega) = \tau^*_{s, X_s^{t,x}}(\omega)$ for a.e. $\omega \in \{\tau^*_{t,x} \ge s\}$ may \emph{not} hold for some $x\in \bX$ and $s,t \in [[0,T]]$ with $s > t$.
\end{proposition}
\begin{proof}
We demonstrate this by constructing a deterministic example with finite horizon $T = 2$ and showing that $\tau^*_0 \ne \tau^*_1$ with $\tau^*_0 > 0$. In this deterministic example, the state space of the Markov chain $X$ degenerates to a singleton and we thus omit the subscript $x$ in the notation. The payoff functions $F_1$ and $F_2$ are determined by a pair of stopping time $\tau, \rho \in \{0,1,2\}$ and two sets of functions, $f_i(t), g_i(t)$ and $h_i(t)$ with $t\in \{0,1,2\}$, for $i \in \{1,2\}$. Specifically, both players' payoffs can be summarized by the following payoff matrices.

\[ F_1 = 
\begin{blockarray}{cccc}
 \rho = 0 & \rho = 1 & \rho = 2& \\
\begin{block}{(ccc)c}
 h_1(0) & f_1(0) & f_1(0)& \tau = 0 & \\
 g_1(0) & h_1(1) & f_1(1)& \tau = 1 &\\
 g_1(0) & g_1(1) & h_1(2)& \tau = 2 &\\
\end{block}
\end{blockarray}
\quad \quad \quad \quad \quad \quad F_2 = 
\begin{blockarray}{cccc}
 \rho = 0 & \rho = 1 & \rho = 2& \\
\begin{block}{(ccc)c}
 h_2(0) & g_2(0) & g_2(0)& \tau = 0 & \\
 f_2(0) & h_2(1) & g_2(1)& \tau = 1 &\\
 f_2(0) & f_2(1) & h_2(2)& \tau = 2 &\\
\end{block}
\end{blockarray}
 \]
 Suppose the following conditions hold (see Remark \ref{r4} for such an example):
 \begin{enumerate}
	\item[(i).] $g_2(0)>h_2(0)$, $h_2(2) > f_2(0) > h_2(1) > g_2(1)$ and $h_2(2) > f_2(1)$;
	\item[(ii).]  $h_1(1) > h_1(2) > \max\{f_1(0),g_1(0)\}$.
\end{enumerate}

At $t = 0$, the optimal response function $\rho^*_0(\tau)$ and the leader's value fuction $V_0(\tau)$ for each $\tau \in \{0,1,2\}$ are given as follows.
 	\begin{equation*}
 	\begin{split}
 	&g_2(0) > h_2(0) \Rightarrow \rho^*_0(0) = 1 \Rightarrow V_0(0) = f_1(0);\\ 
 	&f_2(0) > h_2(1) > g_2(1) \Rightarrow \rho^*_0(1) = 0 \Rightarrow  V_0(1) = g_1(0);\\ 
 	&h_2(2) > \max\{f_2(1) , f_2(0)\} \Rightarrow \rho^*_0(2) = 2 \Rightarrow  V_0(2) = h_1(2).
 	\end{split}
 	\end{equation*}

 Since $h_1(2) > \max\{f_1(0),g_1(0)\}$, by Definition~\ref{def:precom_finite_pure}, the precommitment strategy at $t = 0$ is $\tau^*_0 = 2$. Since $\tau^*_0 > 0$ and the follower's response $\rho^*_0(2) = 2 >0$, the game will proceeds to the next period $t = 1$.

 At $t = 1$, the leader and the follower repeat the above reasoning process and come to a different conclusion.
 	\begin{equation*}
 	\begin{split}
 	&h_2(1) > g_2(1) \Rightarrow \rho^*_1(1) = 1 \Rightarrow  V_1(1) = h_1(1);\\ 
 	&h_2(2) > f_2(1) \Rightarrow \rho^*_1(2) = 2 \Rightarrow  V_1(2) = h_1(2).
 	\end{split}
 	\end{equation*}  

 	Since $h_1(1) > h_1(2)$, the precommitment strategy at $t = 1$ is $\tau^*_1 = 1$, which is inconsistent with $\tau^*_0 = 2$ at $t = 0$. 

 	In this example, the leader's precommitment strategy at time $t = 0$ is to continue until the terminal time. However, upon reaching $t = 1$, the leader has an incentive to deviate from the initial plan, as stopping immediately yields a higher payoff.
\end{proof}

In the deterministic example above, time inconsistency arises inherently from the Stackelberg structure of the game. In contrast to the precommitment strategy, an \emph{equilibrium} strategy accounts for the changing preferences of future selves. By interpreting the optimal stopping problem as a game among  multiple temporal selves---each representing the decision maker at a different point in time. An equilibrium strategy is one in which the current self has no incentive to deviate, assuming all future selves adhere to it. We assume that each temporal self at $t \in [[0,T]]$ adopts a pure Markov stopping policy that determines when to stop, defined as follows.
\begin{definition} \label{def:pure_markov_policy}
A pure Makov stopping policy is defined by a function $p: [[0,T]] \times \bX \to \{0,1\}$, and for each $t \in [[0,T]]$ and $x \in \bX$, the induced stopping time $\tau^p_{t,x}$ is given by 
\begin{equation*}
\tau^p_{t,x} := \inf \{s \ge t \mid p(s, X^{t,x}_s) = 1\}.
\end{equation*}
We also require $p(T, \cdot) = 1$ so that the game always terminates at time $T$. The set of pure Markov stopping policies is given by $\{0,1\}^{[[0,T]] \times \bX}$.
\end{definition}

We formalize the notion of no-regret strategy as a subgame perfect Nash equilibrium in the leader's intrapersonal game across time as follows. 

\begin{definition} \label{def:consistent_equilibrium} A stopping policy $p^* \in \{0,1\}^{[[0,T]] \times \bX}$ is called a (time-consistent) pure equilibrium strategy if for each $t \in [[0, T-1]]$ and $x \in \bX$,
\begin{equation*}
V_{t,x}(\tau_{t,x}^{p^*}) = \sup_{p\in \{0,1\}^{[[0,T]] \times \bX}} V_{t,x}(\tau_{t,x}^{p \oplus_t p^*}),
\end{equation*}
where $p \oplus_t p^* \in \{0,1\}^{[[0,T]] \times \bX}$ is the policy that follows $p$ before time $t$ and $p^*$ after time $t$. Specifically, it is given by
\begin{equation*}
p \oplus_t p^*: (s,x) \mapsto \begin{cases}
p(s,x) & \text{ if } s \le t;\\ 
p^*(s,x) & \text{ if } s > t.
\end{cases}
\end{equation*}
\end{definition}

\begin{remark}
    The condition $V_{t,x}(\tau_{t,x}^{p^*}) = \sup_{p\in \{0,1\}^{[[0,T]] \times \bX}} V_{t,x}(\tau_{t,x}^{p \oplus_t p^*})$ precisely captures the equilibrium idea: at any time $t$,  the leader has no incentive to deviate from the policy $p^*$, provided that she reverts to $p^*$ from time $t+1$ onward. In continuous time, however, the appropriate notion of equilibrium strategy is considerably more subtle; see, for example, \cite{https://doi.org/10.1111/mafi.12293}. There, equilibrium conditions are typically characterized through local (first-order) variational arguments rather than discrete one-step deviations. For this reason, the techniques developed in the present discrete-time framework do not extend directly to the continuous-time setting, which we leave for future research.
\end{remark}

In the finite-horizon setting, a pure equilibrium strategy can be constructed via backward induction. By definition, an equilibrium strategy is Markovian and, in general, differs from a precommitment strategy. In what follows, we compare these two notions with the concept of a Nash equilibrium in the classical Dynkin game, which has been extensively studied. For completeness, we first provide the definition of a Nash equilibrium in the Dynkin version of our game.

\begin{definition} For each $t \in [[0,T]]$ and $x \in \bX$, a pair of stopping times $(\tau^*,\rho^*) = (\tau_{t,x}^*,\rho_{t,x}^*)$ is called a Nash equilibrium (w.r.t.~the initial time $t$) in the following Dynkin game,
\begin{align*}
& \text{Player 1: } \sup_{\tau \in \bT_t} J_1(t,x; \tau, \rho), \\ 
& \text{Player 2: } \sup_{\rho \in \bT_t} J_2(t,x; \tau, \rho),
\end{align*}
if 
\begin{equation*}
J_1(t,x; \tau^*, \rho^*) = \sup_{\tau \in \bT_t} J_1(t,x; \tau, \rho^*)  \quad \text{ and } \quad  J_2(t,x; \tau^*, \rho^*) =  \sup_{\rho \in \bT_t} J_2(t,x; \tau^*, \rho).
\end{equation*}
\end{definition}

\begin{proposition} \label{prop:distinct_notions}
The precommitment strategy, equilibrium strategy, and Nash equilibrium represent distinct solution concepts.
\end{proposition}

\begin{proof}
Again we prove this proposition by a deterministic example. Consider the example in the proof of Proposition~\ref{prop:time_inconsistency}. In addition to Condition (i) and (ii), we assume that 
\begin{enumerate}
	\item[(iii).] $f_1(0)> g_1(0)$;
	\item[(iv).] $g_1(0)>h_1(0)$;
	\item[(v).] $h_1(1)>g_1(1)$ and $f_1(1) > h_1(2)$.
\end{enumerate}
(See Remark \ref{r4} for such an example.) Let us first construct an equilibrium strategy $p^* \in \{0,1\}^{[[0,T]] \times \bX}$ via backward induction. The set of pure Markov stopping policies degenerates to $\{0,1\}^{[[0,T]]}$ in the deterministic scenario. The associated stopping time $\tau^p_t = \inf\{s \ge t \mid p(s) = 1\}$ is the first time that $p$ equals one.

At $t = 1$, $p^*(1)$ satisfies 
\begin{equation*}
V_1(\tau_1^{p^*}) = \max\{V_1(1), V_1(2)\}.
\end{equation*}
By the same reasoning in precommitment strategy for this one-period game, $V_1(1) = h_1(1)>h_1(2) = V_1(2)$, and thus $p^*(1) = 1$ and $\tau_1^{p^*} = 1$. The follower's optimal response is $\rho^*_{1}(1) = 1$. 

At $t = 0$, $p^*(0)$ satisfies
\begin{equation*}
V_0(\tau_0^{p^*}) = \max\{V_0(0), V_0(\tau_1^{p^*})\}.
\end{equation*}
Since $V_0(0) = f_1(0)$, $V_0(\tau_1^{p^*}) = V_0(1) = g_1(0)$ and $f_1(0)> g_1(0)$ by Condition (iii), we have $V_0(\tau_0^{p^*}) = f_1(0)$, $p^*(0) = 1$ and $\tau_0^{p^*} = 0$.

Therefore $p^*(t) \equiv 1$ for all $t \in \{0,1,2\}$ is the equilibrium strategy, i.e., the equilibrium strategy is to stop immediately at each time period. In particular, w.r.t.~the initial time $0$, the leader will stop at time $0$ under the equilibrium strategy, but will continue until the end of the game under the precommitment strategy.

Next let us study the Nash equilibrium in this example. The best response of Player 2 to the strategy of Player 1 at $t = 0$ has been discussed in the proof of Proposition~\ref{prop:time_inconsistency}, and is given by 
\begin{equation*}
\rho^*_0(0) = 1, \rho^*_0(1) = 0, \rho^*_0(2) = 2.
\end{equation*}
For each $\rho \in \{0,1,2\}$, the optimal response of player 1 is the following.
\begin{equation*}
 	\begin{split}
 	&g_1(0) > h_1(0) \Rightarrow \tau^*_0(0) \in \{1,2\};\\ 
 	&h_1(1) > \max\{g_1(1), f_1(0) \} \Rightarrow \tau_0^*(1) = 1;\\ 
 	&f_1(1) > h_1(2) > f_1(0) \Rightarrow \tau_0^*(2) = 1.
 	\end{split}
 \end{equation*}
 By definition, $(\tau^*,\rho^*)$ is a Nash equilibrium if $\tau^* \in \rho_0^*(\tau^*)$ and $\rho^* \in \tau_0^*(\rho^*)$. Then we have that $(\tau^* = 1, \rho^* = 0)$ is a Nash equilibrium (w.r.t.~the initial time $0$), which is distinct from the pre-commitment strategy ($\tau^* = 2, \rho^* = 2$) and the equilibrium strategy ($\tau^* = 0, \rho^* = 1$).
\end{proof}

\begin{remark}
Notice that if condition (iv) does not hold, then $\tau_0^*(0) = 0$ and a Nash equilibrium in the nonzero-sum Dynkin game does not exists, whereas the precommitment and equilibrium strategies remain unchanged. This again highlights the difference between these solution concepts.
\end{remark}

\begin{remark}\label{r4}
One can easily find	an example that satisfies all conditions (i-vi). For example, 
 	 \begin{equation}\label{eg1}
 	 \begin{split}
 	 &f_1(0) = 3, f_1(1) = 5, f_1(2) = 4, f_2(0) = 3, f_2(1) = 3, f_2(2) = 4,\\
 	 &g_1(0) = 2, g_1(1) = 4, g_1(2) = 4, g_2(0) = 3, g_2(1) = 1, g_2(2) = 4,\\
 	 &h_1(0) = 1, h_1(1) = 5, h_1(2) = 4,h_2(0) = 2, h_2(1) = 2, h_2(2) = 4.
 	 \end{split}
 	 \end{equation}

     In this project with  termination and withdrawal rights, if neither the sponsor and the developer exits before date $2$, the project is completed and both the sponsor and the developer receive payoff $4$. At time $t= 0$, if the sponsor promises to continue until time $2$, the developer will also continue until time $2$ since withdrawal earlier does not lead to higher outside opportunity; if the sponsor announces that she will terminate the project at time $1$, then the developer has a larger outside opportunity (with payoff 3) to withdrawal at the initial time, which makes the sponsor's payoff only 1. Taking this into consideration, the sponsor's initial plan is to terminate at time $2$. However at time $1$, the sponsor has an attractive liquidation opportunity worth 5, exceeding its contractual completion payoff $4$ since the developer's outside opportunity now deteriorates from $3$ to $1$.

\end{remark}

\subsection{Comparison with notions in classical game theory}\label{sec:concepts}

In classical dynamic Stackelberg games, three solution concepts are commonly considered: open-loop equilibrium, closed-loop equilibrium, and feedback equilibrium; see \cite{MR475954, simaan1973additional}. In this subsection, we introduce stopping-time counterparts of these concepts and contrast them with the two solution notions discussed earlier, namely precommitment and equilibrium.

\medskip

\noindent
\textbf{Open-loop equilibrium.}
We first describe the admissible policy sets of the two players. Define
\[
\cP := \left\{P: \omega \mapsto (P_0(\omega_0), P_1(\omega_0,\omega_1), \ldots, P_T(\omega)) \in \{0,1\}^{T+1} \ \middle|\ 
P_t \le P_{t+1}\ \forall t,\ P_T = 1 \right\}.
\]
There is a one-to-one correspondence between $\cP$ and the set $\bT$ of stopping times. Indeed, for any $\tau \in \bT$, the process $(\mathbf{1}_{\{\tau \le t\}})_{t\in [[0,T]]}$ belongs to $\cP$, and for any $P \in \cP$, the stopping time
\[
\tau^P := \inf\{t \ge 0 : P_t = 1\}
\]
belongs to $\bT$.

At time $0$, the leader’s open-loop policy set is $\cP$. The follower’s policy depends on the leader’s choice and is therefore described by a mapping $ P \mapsto Q(P)$. Denote
\[
\cQ := \{ Q : \cP \to \cP \}.
\]

\begin{definition}
A pair $(P^*,Q^*) \in \cP \times \cQ$ is called an open-loop Stackelberg equilibrium at time $0$ if for any initial state $x \in \bX$,
\[
\begin{aligned}
&J_1(0,x;\tau^{P^*},\tau^{Q^*(P^*)}) \ge J_1(0,x;\tau^{P},\tau^{Q^*(P)}), \quad \forall P \in \cP;\\
&J_2(0,x;\tau^{P},\tau^{Q^*(P)}) \ge J_2(0,x;\tau^{P},\tau^{Q(P)}), \quad \forall P \in \cP,\ \forall Q \in \cQ.
\end{aligned}
\]
\end{definition}

If $\tau_0^*$ is a precommitment strategy at time $0$, define
\[
P_t^* = \one_{\{\tau_0^* \le t\}}, 
\qquad 
Q_t^*(P) = \one_{\{\rho^*(\tau^P) \le t\}}, \quad \forall t \in [[0,T]]
\]
where $\rho^*(\tau^P)$ denotes the follower’s optimal response to $\tau^P$. Then $(P^*,Q^*)$ forms an open-loop Stackelberg equilibrium. On the other hand, an open-loop Stackelberg equilibrium does not in general impose a specific tie-breaking rule (such as choosing the earliest optimal stopping time), and can therefore be viewed as a broader concept than precommitment strategy.

\medskip

\noindent
\textbf{Closed-loop equilibrium.}
The set of pure Markov stopping policies corresponds to closed-loop policies for the leader:
\[
\cP^{cl} := \{0,1\}^{[[0,T]] \times \bX}:=\{p: [[0,T]] \times \bX \ni (t,x) \mapsto p(t,x) \in \{0,1\}\}.
\]
The follower’s closed-loop policy set is
\[
\cQ^{cl} := \{ q : \cP^{cl} \to \cP^{cl} \}.
\]

\begin{definition}
A pair $(p^*,q^*) \in \cP^{cl} \times \cQ^{cl}$ is called a closed-loop Stackelberg equilibrium if for any $x \in \bX$,
\[
\begin{aligned}
&J_1(0,x;\tau^{p^*},\tau^{q^*(p^*)}) \ge J_1(0,x;\tau^{p},\tau^{q^*(p)}), \quad \forall p \in \cP^{cl};\\
&J_2(0,x;\tau^{p},\tau^{q^*(p)}) \ge J_2(0,x;\tau^{p},\tau^{q(p)}), \quad \forall p \in \cP^{cl},\ \forall q \in \cQ^{cl}.
\end{aligned}
\]
\end{definition}

In deterministic settings as in the previous subsection example, open-loop and closed-loop equilibria coincide. However, in general stochastic environments, we have
$\cP^{cl} \times \cQ^{cl} \subset \cP \times \cQ$.

\medskip

\noindent
\textbf{Feedback equilibrium.}
Although the equilibrium strategy $p^*$ in Definition~\ref{def:consistent_equilibrium} is Markovian, the pair $(p^*,q^*)$ does not form a closed-loop equilibrium in the above sense. Instead, it corresponds to a \emph{feedback Stackelberg equilibrium}, where the follower’s action at time $t$ depends on both the current state and the leader’s current action.

\begin{definition}
    Denote \[
    \cQ^{fb}:= \{q : [[0,T]] \times \bX \times \{0,1\} \ni (t,x, p) \to q(t,x,p) \in \{0,1\}  \mid q(T,\cdot, \cdot)  = 1\}.
    \]
    A pair $(p^*,q^*) \in \cP^{cl} \times \cQ^{fd}$ is a feedback Stackelberg equilibrium for the Stackelberg stopping game at time $0$ if for any $t \in [[0,T-1]]$ and any $x \in \bX$,\[
    \begin{split}
        q^*(t,x,0) \in \arg\max_{q \in \{0,1\}} & q f_2(x) +  (1-q) \bE_{t,x}[F^*_2(t+1,X_1)]\\
        q^*(t,x,1) \in \arg\max_{q \in \{0,1\}} & q  h_2(x) +  (1-q) g_2(x)\\
        p^*(t,x) \in \arg\max_{p \in \{0,1\}} & p \left[ q^*(t,x,p) h_1(x) + (1- q^*(t,x,p))f_1(x)
        \right] \\
        &+ (1-p) \left[ q^*(t,x,p) g_1(x) + (1- q^*(t,x,p))\bE_{t,x}[F^*_1(t+1,X_1)] \right]
        \end{split}
    \]

    where \[
    \begin{split}
         F^*_2(t,x) = &  p^*_{t,x} \max\{h_2(x), g_2(x)\} + (1-p^*(t,x)) \max\{f_2(x), \bE_{t,x}[ F^*_2(t+1,X_1)] \} ,\quad \forall t \in [[0,T-1]];\\       
         F^*_1(t,x) = & \max_{p \in \{0,1\}} p \left[ q^*(t,x,p) h_1(x) + (1- q^*(t,x,p))f_1(x)
        \right] \\
        &\quad \quad \quad+ (1-p) \left[ q^*(t,x,p) g_1(x) + (1- q^*(t,x,p))\bE_{t,x}[F^*_1(t+1,X_1)] \right] ,\quad \forall t \in [[0,T-1]];\\
    F^*_i(T,x) =& h_i(x), \quad  i \in \{1,2\}.
    \end{split}
    \]
\end{definition}

\medskip

In the finite-horizon case, the feedback Stackelberg equilibrium is obtained via backward induction and coincides with the equilibrium notion introduced in Definition~\ref{def:consistent_equilibrium}.

\begin{remark} \label{rmk:distinction}
The notion of equilibrium in Definition~\ref{def:consistent_equilibrium} should not be interpreted as a subgame perfect Nash equilibrium between the leader and the follower. Instead, it corresponds to a subgame perfect equilibrium in an intrapersonal dynamic game among the leader’s successive selves, and thus captures the requirement of time consistency for the leader’s strategy. In classical terminology, this notion is most closely related to a feedback Stackelberg equilibrium. It differs from both open-loop and closed-loop Stackelberg equilibria and does not, in general, coincide with subgame perfect Nash equilibrium in two-player dynamic games.

It is worth emphasizing that this distinction is substantive, not merely terminological. In classical game theory, every subgame perfect Nash equilibrium is also a Nash equilibrium. In contrast, in the Stackelberg stopping framework considered here, a precommitment strategy can be viewed as a Nash equilibrium of the underlying two-player game, since neither player has an incentive to deviate given the other’s strategy. However, an equilibrium in the sense of Definition~\ref{def:consistent_equilibrium} need not be a Nash equilibrium of the two-player game: the leader may have an incentive at time $0$ to deviate to the precommitment strategy. Consequently, Proposition~\ref{prop:distinct_notions} is not a restatement of the classical result that a Nash equilibrium may fail to be subgame perfect; rather, it reflects the fundamentally different role played by time consistency in Stackelberg stopping problems.
\end{remark}

\subsection{Randomized Strategies}\label{sec:finite_random}

In this subsection, we study relaxed/randomized strategies within the Stackelberg stopping framework. We formalize these notions in the finite-horizon setting and illustrate them using the deterministic example introduced earlier. A more detailed theoretical analysis is developed in the infinite-horizon setup in Section 3.

Let $Y = (Y_t)_{t \in \bN_0}$ be a sequence of independent uniformly distributed random variables on interval $[0,1]$, which are also independent of the Markov chain $X$. For each sequence  $p = (p_0,p_1,p_2,\cdots) \in [0,1]^{\infty}$, define
\begin{equation*}
\tau^{p,Y} := \inf \{t \ge 0 \mid Y_t \le p_t\}. 
\end{equation*}
Then $\tau^{p,Y}$ is a stopping time w.r.t.~the filtration generated by $Y$. Each $p_t \in [0,1]$ represents the probability that the process will stop at $t$, given that it has not stopped before time $t$.

Now we consider randomized stopping policies that map the path $\omega = (\omega_0, \cdots, \omega_T)\in \Omega$ to a sequence $p = (p_0, \cdots, p_T) \in [0,1]^{T+1}$.

\begin{definition} \label{def:finite_random_policy}
A randomized stopping policy on time interval $[[0,T]]$ is an adapted process 
\begin{equation*}
P = P(\omega) := \left(P_0(\omega_0), \cdots, P_{T-1}(\omega_0, \cdots, \omega_{T-1}), P_T(\omega)\right), \quad \forall \omega \in \Omega,
\end{equation*}
with $P_T \equiv  1$.
Here the random variable $P_t$ represents the probability that the process will stop at time $t$, given that it has not stopped at time $t-1$. This conditional probability $P_t$ depends on the trajectory $\omega_0, \cdots, \omega_t$ of the Markov chain $X$. 

The induced stopping time of policy $P$ is given by 
\begin{equation*}
\tau^{P}:= \tau^{P, Y} := \inf \{t \ge 0 \mid Y_t \le P_t\},
\end{equation*}
which is a stopping time w.r.t.~the filtration generated by $X$ and $Y$. We denote the set of all randomized stopping policies on time interval $[[0,T]]$ by $\cR_{T}$.
\end{definition}

\begin{definition} \label{def:finite_random_markov_policy}
If each $P_t$ depends only on the current state $\omega_t \in \bX$ rather than the full path $\omega_0, \cdots, \omega_t$, then the stopping policy $P \in \cR_T$ degenerates to a function $p: [[0,T]] \times \bX \to [0,1]$, which is a natural generalization of the pure Markov stopping policy defined in Definition~\ref{def:pure_markov_policy}. Thus we call it a randomized Markov stopping policy. The set of all randomized Markov stopping policies on time interval $[[0,T]]$ is given by $[0,1]^{(T+1)\times N}$.

\end{definition}

\begin{remark}\label{rmk:shift_operator}
Each $P \in \cR_T$ can be described by a $(T+1)$-period multinomial tree with $N$ branches per time step. It can be decomposed into $N$ subtrees as follows. \begin{equation*}
	P = \left\{P_0(x) \oplus \theta_{x}\circ P\right\}_{x\in \bX},
\end{equation*}
where $P_0(x) \in [0,1]$ is the probability that the leader will stop at the current step given that $X_0 = x$, and $\theta_{x}\circ P = \left\{P_1(x,\cdot), P_2(x, \cdot, \cdot), \cdots, P_T\right\} \in \cR_{1:T}$ is the randomized path-dependent stopping strategy for the next step, given that the leader does not stop at $X_0 = x$. We use notation $\cR_{t:T}$ to denote the set of all randomized stopping policies on time interval $[[t,T]]$.

Here, $\theta_x$ denotes the one-step shift operator applied to a policy $P$, given that the current state is $x$. For simplicity, we denote the $t$-step shift operator by 
\begin{equation*}
\theta_{\omega_0, \omega_1, \cdots, \omega_{t-1}} := \theta_{\omega_{t-1}}\circ \cdots \circ \theta_{\omega_0},
\end{equation*}
which corresponds to restricting the policy to the subinterval $[[t,T]]$, conditional on the observed history $\omega_0, \omega_1, \cdots, \omega_{t-1}$. When it is unnecessary to emphasize the specific trajectory, we refer to the shifted policy on the time interval $[[t,T]]$, namely $\theta_{\omega_0, \cdots, \omega_{t-1}} \circ P$, simply as $P^{t}$ for notational convenience. 
\end{remark}

We now allow both players to adopt randomized stopping policies. To model this, let $Y = (Y_t)_{t \in \bN_0}$ and $Y' = (Y'_t)_{t \in \bN_0}$ be sequences of independent random variables, each uniformly distributed on $[0,1]$, and assume that both $Y$ and $Y'$ are independent of the underlying Markov chain $X$ and of each other. We denote the stopping policies of the leader and the follower by $P \in \cR_T$ and $Q \in \cR_T$, respectively. Then stopping times induced by these policies are given by 
\begin{equation*}
\tau^P = \tau^{P, Y}, \quad \tau^Q = \tau^{Q,Y'}.
\end{equation*} 
Intuitively, the sequences  $Y$ and $Y'$ serve as internal randomization devices that allow each player to make probabilistic stopping decisions in a way that can depend on both the history of the process X and their own private source of randomness. 

In the Stackelberg stopping game, the leader announces her policy $P \in \cR_T$ first, then the follower will respond after observing whether the leader stops at time $0$. That is the follower's response depends on both $P$ and the realization of random variable $Y_0$. If the leader stops at $t = 0$, then the follower solves the following problem. 
\begin{equation*}
W_S(0,x) := \sup_{q \in [0,1] } q h_2(0, x) + (1-q) g_2(0,x),
\end{equation*}
and the optimal response is given by 
\begin{equation*}
Q_S(0,x) = \one_{\{h_2(0,x) \ge g_2(0,x)\}}.
\end{equation*}
Here the subscript `S' stands for \emph{stop}. In what follows, we also use subscript `C' to indicate \emph{continue}.

If the leader does not stop at $t = 0$, then the follower solves the following problem.
\begin{equation*}
W_C(0,x; \theta_x \circ P) := \sup_{Q \in \cR_T } \bE_{t,x}[F_2(\tau^P, \tau^Q) \mid \tau^P > 0].
\end{equation*}

Note that the utility $W_C$ is a function of $\theta_x \circ P$, since the stopping time $\tau^P$ does not depend on $P_0(\cdot)$ given that $\tau^P > 0$. Moreover, for this fixed stopping policy $\theta \circ P \in \cR_{1:T}$, the utility of Player 2 satisfies the following Bellman's equation.

\begin{equation*}
W_C(0,x; \theta_x \circ P) = \sup_{q \in [0,1]} q f_2(0,x) + (1-q) \bE_{0,x}[W(1,X_1;\theta_x \circ P)], 
\end{equation*}
where $W(t,x;\tilde P)$ denotes the total utility of the follower at time $t$, given $X_t=x$ and the leader's strategy $\tilde P\in\mathcal{R}_{t:T}$. The optimal response is given by 
\begin{equation*}
Q_C(0,x; \theta_x \circ P) =  \one_{\{f_2(0,x) \ge \bE_{0,x}[W(1,X_1; \theta_x \circ P)]\}}.
\end{equation*}
The follower's value at time $0$ given $X_0=x$ is given by
\begin{equation*}
W(0,x; P) = P_0(x)W_S(0,x) +(1-P_0(x))W_C(0,x,\theta_x \circ P).
\end{equation*}

Repeating the above process, we have that for each $t \in [[0,T-1]]$, $y \in \bX$ and $P\in\mathcal{R}_T$, \[
\begin{split}
&W(t,y; P^{t}) = P_t(y)W_S(t,y) +(1-P_t(y))W_C(t,y,P^{t+1});\\
&W_S(t,y) = \max\{h_2(t,y), g_2(t,y)\}; \quad \quad   W_C(t,y; P^{t+1}) = \max\{f_2(t,y), \bE_{t,y}[W(t+1,X_{t+1}; P^{t+1})] \};\\
&Q_S(t,y) =  \one_{\{h_2(t,y) \ge g_2(t,y)\}};  \quad\quad  Q_C(t,y; P^{t+1})=  \one_{\{f_2(t,y) \ge \bE_{t,y}[W(t+1,X_{t+1}; P^{t+1})]\}}.
\end{split}\] 

The above reasoning shows that the follower's optimal response consists of two parts: $Q_S$ and $Q_C$. The stop-scenario response $Q_S$ depends only on time $t$ and current state $X_t$, while the continue-scenario response $Q_C$ depends on the leader's remaining policy $P^{t+1}$ in addition to  time $t$ and current state $X_t$. For simplicity, we write $Q^* = Q^*(P) = (Q_S, Q_C(P))$ as the total optimal response of the follower. Then the stopping time induced by $Q^*$ is given by
\begin{equation*}
\tau^{Q^*} := \inf\left\{t \ge 0 \mid Y'_t \le Q_S(t,X_t) \one_{\{\tau^P \le t\}} + Q_C(t,X_t; P^{t+1})\one_{\{\tau^P > t\}}\right\},
\end{equation*}

For leader's stopping policy $P$ defined on time interval $[[t,T]]$, the optimal response $Q^*$ by the follower can be defined in the same manner as before. To avoid cumbersome notation, we continue to use the symbol $Q^*$ to denote this response. With a slight abuse of notation, we define the value function
\begin{equation*}
V_{t,x}(P) := \bE_{t,x}[F_1(\tau^P, \tau^{Q^*(P)})], \quad  t\in [[0,T]], \,  x \in \bX,  \, P\in \cR_{t:T}.
\end{equation*}

As in the pure strategy case, the Stackelberg game can be formulated as an optimal control problem from the perspective of the leader:
\begin{equation} \label{eq:optimal_control_problm}
\sup_{P \in \cR_T} V_{0,x}(P).
\end{equation}

\begin{definition}
A randomized stopping policy $P^* \in \cR_T$ is called a randomized precommitment strategy if 
\begin{equation*}
V_{0,x}(P^*) = \sup_{P \in \cR_T} V_{0,x}(P),  \quad \forall x \in \bX.
\end{equation*}
\end{definition}

 Since the set of randomized stopping policies strictly contains the set of pure stopping policies, it is natural to expect that a randomized policy may yield a strictly higher utility for the leader than any pure policy does. On the other hand, as we will see in the proof of Proposition~\ref{prop:precomit_random_nonexist}, a randomized precommitment strategy may not exist. Intuitively, this nonexistence arises from the discontinuity of the follower’s best response mapping $Q^*(P)$ w.r.t.~the leader’s policy $P$. 

\begin{proposition} \label{prop:precomit_random_nonexist}
A randomized precommitment strategy for Problem~\eqref{eq:optimal_control_problm} may not exist. 
\end{proposition}

\begin{proof}
We establish this by presenting a counterexample. Specifically, we revisit the simplified deterministic example from Section~\ref{sec:finte_pure} with the coefficients given in \eqref{eg1}. The set of randomized stopping policies now degenerates to $[0,1]^2$ and the leader aims to find $(P_0, P_1) \in [0,1]^2$ that maximizes her utility. Denote the uility of the leader (resp. follower) when the leader has stopped and has not stopped at time $t$ by $V_S(t)$ and $V_C(t)$ (resp. $W_S(t)$ and $W_C(t)$), respectively. Denote the expected utility of the leader (resp. follower) at time $t$ by $V(t)$ (resp. $W(t)$). We have 
\begin{equation*}
V(t) = P_t V_S(t) +(1-P_t) V_C(t) \text{ and  } W(t) = P_t W_S(t) +(1-P_t) W_C(t).
\end{equation*}

If the leader stops at $t = 0$, the follower chooses between $h_2(0)$ and $g_2(0)$. Since $h_2(0)<g_2(0)$, the follower will continue and the utility of the leader is $V_S(0) = f_1(0) = 3$.

If the leader does not stop at $t = 0$, the follower chooses between $f_2(0)$ and the expected payoff at $t = 1$, which is $W(1) = P_1 W_S(1) +(1-P_1) W_C(1)$. Then the leader's utility is 
\begin{equation*}
V_C(0) = g_1(0) \one_{\{f_2(0) \ge W(1)\}} + V(1) \one_{\{f_2(0) < W(1)\}}.
\end{equation*}

When both the leader and the follower proceed to $t = 1$, since $h_2(1) > g_2(1)$, the follower will stop if the leader stops at $t = 1$. Thus $W_S(1) = h_2(1)$ and $V_S(1) = h_1(1)$. If the leader continues at $t = 1$, the follower chooses between $f_2(1)$ and $h_2(2)$, which implies $W_C(1) = h_2(2)$ and $V_C(1) = h_1(2)$.

Therefore,
\begin{equation}\label{eq:v_not_continuous}
\begin{split}
V_C(0) = V_C(0; P_1)& = g_1(0) \one_{\{f_2(0) \ge W(1)\}} + (P_1 h_1(1) +(1-P_1)h_1(2)) \one_{\{f_2(0) < W(1)\}}\\ 
& = 2 \cdot \one_{\{3 \ge 2P_1+4(1-P_1) \}} + (5P_1 +4(1-P_1)) \one_{\{3 < 2P_1+4(1-P_1)\}}.
\end{split}
\end{equation}
Notice that $V_C(0; P_1)$ has a jump at $P_1 = 0.5$. We have $V_C(0, 0.5) = 2$, while $\lim_{P_1 \to 0.5^-} V_C(0, P_1) = 4.5$, which is higher than the payoff when using pure precommitment strategy, $F_1(2,2) = 4$.
\end{proof}

\begin{remark}\label{r1}
The above example shows that a randomized strategy may yield strictly higher utility than the pure precommitment strategy does, although the randomized precommitment strategy may not be attainable.
\end{remark}

In contrast to the randomized precommitment strategy, the existence of a randomized equilibrium strategy---a generalization of the pure equilibrium strategy, as defined in the following paragraph---is guaranteed in the finite-horizon setting.

\begin{definition}
A randomized Markov stopping policy $p \in [0,1]^{(T+1)\times N}$ is called a randomized equilibrium strategy if for each $ t \in [[0, T-1]]$,
\begin{equation*}
V_{t,x}(p^t) \ge  V_{t,x}(p'_t\oplus p^{t+1}), \quad \forall x \in \bX, \, \forall p'_t = \{p'_t(y)\}_{y \in \bX} \in [0,1]^{N}
\end{equation*}
where 
\begin{equation*}
p^t := (p(t, \cdot), \cdots, p(T,\cdot))
\end{equation*}
is the restriction of policy $p$ on time interval $[t,T]$, and 
\begin{equation*}
p_t' \oplus p^{t+1} := (p'_t(\cdot), p(t+1, \cdot), \cdots, p(T,\cdot))
\end{equation*}
is a one-step deviation of policy $p^t$ on time interval $[t,T]$.
\end{definition}

\begin{remark}
In the finite-horizon case, the randomized equilibrium strategy can be derived via backward induction. At each time $t$ with $X_t = x$, the leader solves a linear optimization problem of the form $\sup_{p_t(x) \in [0,1]} p_t(x) V_S(t,x) + (1-p_t(x))V_C(t,x; p^{t+1})$. As a result, the randomized equilibrium strategy degenerates to a pure equilibrium strategy.

This is no longer true in the infinite-horizon setting; see Section~\ref{sec:infinity_equilibrium} for details.
\end{remark}

\section{The Infinite-Horizon Stackelberg Stopping Game} \label{sec:infinite}

Motivated by the results and examples in the finite-horizon setting, we undertake a comprehensive analysis of both the precommitment and equilibrium strategies with randomization to address time inconsistency in the Stackelberg stopping game under an infinite-horizon framework.  Moreover, we generalize the Markov chain by extending its state space from a finite set to an uncountable one. There are three main reasons for focusing on the infinite-horizon case.

First, the time-homogeneous structure of the Markov chain and the use of exponential discounting simplify the notation and allow for a clearer investigation of the time-inconsistent nature of the Stackelberg game. 

Second, the analysis of precommitment strategies in the infinite-horizon setting closely parallels that of the finite-horizon case, making it a natural extension.

Third, the study of equilibrium strategies in the infinite-horizon case leads to novel and nontrivial challenges, particularly due to the need for fixed-point arguments. In contrast, equilibrium strategies in the finite-horizon setting can be constructed via backward induction, rendering the analysis relatively straightforward.

Let us first state the infinite-horizon set-up. 
Consider a Polish space $\bX$ and denote by $\cB(\bX)$ the family of Borel sets in $\bX$. Let $X = (X_t)_{t\in \bN_0}$ be a time-homogeneous Markov chain taking values in $\bX$. The transition kernel of $X$ is denoted by $\Pi$, i.e., 
\begin{equation*}
    \bP(X_{t+1} \in A | X_t = x) = \int_A \Pi(x, dy), \quad \forall A \in \cB(\bX), x \in \bX, t \in \bN_0.
\end{equation*}
Denote by $\mathbb{P}_x$ and $\mathbb{E}_x$ the measure under which $X_0 = x \in \mathbb{X}$ a.s.~and the associated expectation, respectively. The canonical space of the Markov chain on infinite horizon is denoted by
\begin{equation*}
\Omega := \{\omega = (\omega_0, \omega_1, \cdots) \mid \omega_t \in \bX, \forall t \in \bN_0\},
\end{equation*}
and the filtration generated by $\{X_t\}_{t\in \bN_0}$ is denoted by $\cF^{X}$.

In the Stackelberg stopping game, the leader (Player 1) and the follower (Player 2) choose their stopping strategies to maximize their expected payoff, which depends on the stopping times of both players. Denote the stopping time of the leader and the follower by $\tau$ and $\rho$, respectively. Denote 
\begin{equation} \label{eq:F1F2}
\begin{split}
&F_1(\tau, \rho) := f_1(X_{\tau})\one_{\{\tau<\rho\}} + g_1(X_{\rho})\one_{\{\rho<\tau\}} + h_1(X_\tau)\one_{\{\rho=\tau\}},\\ 
&F_2(\tau, \rho) := f_2(X_{\rho})\one_{\{\tau>\rho\}} + g_2(X_{\tau})\one_{\{\rho>\tau\}} + h_2(X_\rho)\one_{\{\rho=\tau\}},
\end{split}
\end{equation}
where $f_i, g_i, h_i, i \in \{1,2\}$ are Borel measurable real-valued functions on $\bX$. 

For the remainder of the paper, we impose the following two assumptions.

\begin{assumption}\label{asp:reference_prob}
There exists a common reference probability $\mu$ on the measurable space $(\bX, \cB(\bX))$ such that for any $x \in \bX$, the transition kernel $\Pi(x, \cdot)$ admits a density function $\pi(x,\cdot)$ w.r.t.~$\mu$. That is,
\[
\Pi(x, dy) = \pi(x,y)\mu(dy).
\]
By definition, the density function family $\{\pi(x, \cdot)\}_{x\in \bX}$ is uniformly bounded in the space $L^{1}(\bX,\mu)$, with $
 \|\pi(x, \cdot)\|_{L^{1}(\bX,\mu)}  = 1$ for each $x\in \bX$.
\end{assumption}

\begin{assumption} \label{asp:bounded_fgh}
The functions $f_i,g_i,h_i, i \in \{1,2\}$ are bounded, i.e.,
\[
K:=\sup_{x\in \bX} \sum_{i\in \{1,2\}}|f_i(x)| + |g_i(x)| + |h_i(x)| < \infty.
\]
\end{assumption}
As shown in the previous section, the leader-follower structure in the Stackelberg stopping game introduces time inconsistency into the leader’s optimization problem. In the infinite-horizon setting, despite time-homogeneity, it becomes necessary to consider path-dependent stopping policies. We revisit the concept of randomized stopping policy and randomized Markov stopping policy in Definition~\ref{def:finite_random_policy} and Definition~\ref{def:finite_random_markov_policy}, and formulate them in the infinite-horizon framework as follows.

\begin{definition}
A randomized (path-dependent) stopping policy is a stochastic process adapted to $\cF^{X}$,
\begin{equation*}
P = P(\omega) := (P_0(\omega_0), P_1(\omega_0, \omega_1), \cdots, P_t(\omega_0, \cdots, \omega_t), \cdots) \in [0,1]^{\bN_0}, \quad \forall \omega \in \Omega.
\end{equation*}
Here the random variable $P_t$ represents the probability that the process will stop at time $t$, given that it has not stopped by time $t-1$. This conditional probability $P_t$ depends on the trajectory $(\omega_0, \cdots, \omega_t)$ of the Markov chain $X$. Denote the set of all randomized stopping policy by $\cR$.
\end{definition}

We endow the set of randomized stopping policies $\cR$ with the topology induced by the distance function
\begin{equation}\label{eq:distance}
d(P,P'):= \sum_{t=0}^\infty\frac{1}{2^t} \int_{X^{t+1}}|P_t(y_0,y_1,\dotso,y_t)-P'_t(y_0,y_1,\dotso,y_t)| \mu(d y_0) \cdots \mu(d y_t),\quad P,P'\in\mathcal{R}.
\end{equation}
Under such topology, $\cR$ is a Polish (complete separable metric) space.

%

As in the finite-horizon case, we assume that $Y = (Y_t)_{t \in \bN_0}$ and $Y' = (Y'_t)_{t \in \bN_0}$ are the private sources of randomness for the leader and follower, respectively. They are two independent sequences of independent uniformly distributed random variables on $[0,1]$, which are also independent of $X$. The stopping time induced by the stopping policy $P \in \cR$ and the private randomness source $Y$ is given by 
\begin{equation} \label{eq:tau_bigP}
\tau^P = \tau^{P,Y} := \inf\{t\ge 0 \mid Y_t \le P_t(X_0, \cdots, X_t)\},
\end{equation}
which is a stopping time w.r.t.~the filtration generated by $X$ and $Y$.

Since the follower makes his move based on the information of both the leader's strategy and the leader's current state, stopping or continuing.  the follower's strategy $Q$ consists of two parts, $Q_S \in \cR$ and $Q_C \in \cR$, which represent the path-dependent randomized stopping policies when the leader stops and when the leader continues at the current step, respectively. Thus we denote the set of the follower's policies by $\cR^2$. Specifically, the stopping time corresponding to the follower's policy $Q = (Q_S, Q_C) \in \cR^2$ is given by 
\begin{equation} \label{eq:tau_bigQ}
\tau^{Q_S, Q_C} =  \tau^{Q,Y'} :=\inf\{t\ge 0 \mid Y'_t \le Q_{S,t}(X_0, \cdots, X_t) \one_{\{\tau^P \le t\}} + Q_{C,t}(X_0, \cdots, X_t) \one_{\{\tau^P > t\}} \}.
\end{equation}
Note that $\tau^Q$ is a stopping time with respect to the filtration generated by $X,Y'$ and the process $\left(\one_{\{\tau^P \le t\}}\right)_{t \in \bN_0}$.

\begin{definition} \label{def:infinite_markov_policy}
A randomized Markov (stationary) stopping policy is a Borel measurable function $p_\cdot: \bX \to [0,1]$. The value $p_x$ represents the probability that the process will stop at the current time $t$ with current state $X_t = x$, given that it has not stopped by time $t-1$, for all $ t\in \bN_0$ and $x \in \bX$. 
Denote the set of all randomized Markov stopping policies by $\cR_0$.
\end{definition}

Similar to , the induced stopping time of the leader's and the follower's Markov policy $p \in \cR_0$ and $(r,q) \in \cR_0^2$  are given by 
\begin{align}
&\tau^p = \tau^{p,Y} := \inf\{t \ge 0 \mid Y_t \le p_{X_t}\}; \label{eq:tau_small_p}\\ 
&\tau^{r,q} = \tau^{r,q,Y'} :=\inf\{t\ge 0 \mid Y'_t \le r_{X_t} \one_{\{\tau^p \le t\}} + q_{X_t} \one_{\{\tau^p > t\}} \}.\label{eq:tau_small_q}
\end{align}

We continue to use the notation of the one-step shift operator $\theta_x$ and the more general $t$-step shift operator $\theta_{\omega_0, \cdots,\omega_{t-1}}$ as introduced in Remark~\ref{rmk:shift_operator}.

\begin{remark} \label{rmk:shift_operator_inf}
We generalize the observation in Remark~\ref{rmk:shift_operator}. Each $P \in \cR$ can be represented as an infinite-horizon multinomial tree with $|\bX|$ branches per time step. Accordingly, $P$ admits the following decomposition:. \[
	P = \left\{P_0(x) \oplus \theta_{x}\circ P\right\}_{x\in \bX},
\]
where $P_0(\cdot)  \in \cR_0$ specifies the probability that the leader stops at the current step, and for each $x \in \bX$, $\theta_{x}\circ P = \left\{P_1(x,\cdot), P_2(x, \cdot, \cdot), \cdots\right\} \in \cR$ is the randomized stopping policy from the next step onward, conditional on the leader not stopping and $X_0 = x$. 

Note that for each $P \in\cR$, the mapping $ x \mapsto \theta_{x} \circ P$ defines a Borel measurable function from $\bX$ to $\cR$. Conversely, given any Borel measurable mapping \[\mathbf{P}: \bX   \ni  x \mapsto \mathbf{P}(x) = (\mathbf{P}_0(x)(\omega_0, \omega_1, \cdots),\mathbf{P}_1(x)(\omega_0, \omega_1, \cdots), \cdots) \in \cR,\] together with any $P_0 \in \cR_0$, we can construct a stopping policy in $\cR$ by pasting $P_0$ and $\mathbf{P}$ as $\left\{P_0(x) \oplus \mathbf{P}(x)\right\}_{x\in \bX} \in \cR.$
\end{remark}

For convenience, we introduce the notation
\begin{equation*}
    \cR^{\bX} := \{ \textbf{P}: \bX \to \cR  \text{ is Borel measurable.}\}
\end{equation*}
Throughout the remainder of the paper, we use boldface $\mathbf{P}$ to denote elements of $\cR^{\bX}$, $P$ to denote a stopping policy in $\cR$, $P_0$ (or $p$) to denote a Markov stopping policy in $\cR_0$, and $P_0 \oplus \mathbf{P}$ to denote the pasting of $P_0 \in \cR_0$ and $\mathbf{P} \in \cR^{\bX}$.

The proofs of the results in this section are deferred to Section \ref{sec:proof}.


\subsection{Precommitment Strategy}\label{sec:infinity_precommit}

We begin by formulating the Stackelberg stopping game in which both players adopt randomized stopping policies from the set $\cR$. Although the formulation of the Stackelberg stopping game in the infinite-horizon setting closely parallels that of the finite-horizon case, we find it useful to restate the model here to emphasize the time-homogeneous structure and to introduce notation that will be used in the subsequent analysis. With a slight abuse of notation, the objective functions $J_1$ and $J_2$ in Section \ref{sec:infinity_precommit} and Section \ref{sec:infinity_equilibrium} are functions of stopping policies $(P,(Q_S, Q_C))$ and $(p, (r,q))$, respectively, while $J_1$ and $J_2$ are functions of stopping times $(\tau, \rho)$ in Section \ref{sec:finite}.

The infinite-horizon Stackelberg stopping game is characterized by the following objective function pair $(J_1, J_2)$:
\begin{equation} \label{eq:game_precomit}
    \begin{split}   
          \text{Follower: } & \sup_{(Q_S, Q_C) \in \cR ^2} J_2(x, P, (Q_S, Q_C));\\
        \text{Leader: }&  \quad\,\,\, \sup_{P \in \cR } \quad J_1(x, P, (Q_S^*, Q_C^*)),
    \end{split}
\end{equation}
with \begin{align*}
  &J_i(x, P, (Q_S, Q_C)) := \bE_x\left[ \delta_i^{\tau^{P} \wedge \tau^{(Q_S, Q_C)}} F_i(\tau^{P}, \tau^{(Q_S, Q_C)})\right], \quad i \in \{1,2\}, x\in \bX,  ;\\
  & (Q_S^*, Q_C^*) = \arg\max_{(Q_S, Q_C) \in \cR^2} J_2(x, P, (Q_S, Q_C)),
\end{align*}
where $\delta_1, \delta_2 \in (0,1)$ are the discount factors of the leader and the follower, respectively, and $\tau^P$, $\tau^{Q_S,Q_C}$, and $F_i$ are defined in \eqref{eq:tau_bigP}, \eqref{eq:tau_bigQ}, and \eqref{eq:F1F2}, respectively.
    
If we further assume that the follower always chooses the earliest optimal stopping time, then the follower's optimal response $Q^* = (Q^*_S, Q_C^*)$ is a function of the leader's strategy $P$. Denote this response function by $Q^*(P)$, the above Stackelberg stopping game \eqref{eq:game_precomit} can be reduced  to the following optimization problem for the leader,
\begin{equation}\label{eq:leader_problem_infT}
    \sup_{P \in \cR} V(x,P), \text{ with } V(x, P) := J_1(x, P, Q^*(P)).
\end{equation}

As discussed in Section \ref{sec:finite}, the problem \eqref{eq:leader_problem_infT} is time inconsistent. The solution to this problem is defined as precommitment strategy (see Definition \ref{def:infinite_precommit}) and characterized in Proposition \ref{thm:v_bellman}.
\subsubsection{The follower's problem}

Denote by $W: = W(x,P)$ the follower's maximal expected payoff function before observing the leader's state (i.e., stopping or continuing) at the current time, which depends on current state $X_0 = x$ and the leader's policy $P \in \cR$. Then we have 
\begin{equation*}
W(x,P) = P_0(x) W_S+ (1-P_0(x)) W_C,
\end{equation*}
where $W_S$ and $W_C$ are the follower's maximal expected payoff functions when the leader stops and when the leader continues at the current step, respectively. Specifically, 
\begin{align*}
&W_S := \sup_{Q \in \cR^2} \bE_x\left[\delta_2^{\tau^P\wedge\tau^{Q}}F_2(\tau^P, \tau^{Q}) \mid \tau^P = 0\right]; 
\quad \quad W_C := \sup_{Q \in \cR^2} \bE_x\left[\delta_2^{\tau^P\wedge\tau^{Q}}F_2(\tau^P, \tau^{Q}) \mid \tau^P > 0\right].
\end{align*}

Again we assume that the follower always chooses the earliest stopping time that attains the maximal expected payoffs $W_S$ and $W_C$. By the definition of $F_2$, we have that $W_S$ is a function of current state $x$:
\begin{equation} \label{eq:W_S}
W_S = W_S(x) :=  \sup_{Q_{S,0}(x) \in [0,1]}  Q_{S,0}(x)h_2(x) + (1- Q_{S,0}(x))g_2(x) =  \max \{h_2(x), g_2(x)\},
\end{equation}
of which the optimal response is given by $Q^*_S(x) = \one_{\{h_2(x) \ge g_2(x)\}}.$
Consequently, the follower only needs to solve\[
W_C := \sup_{Q_C \in \cR} \bE_x\left[\delta_2^{\tau^P\wedge\tau^{Q^*_S, Q_C}}F_2(\tau^P, \tau^{Q^*_S, Q_C}) \mid \tau^P > 0\right].
\]
Given $\tau^P > 0$, the quantity $P_0$ will not affect the follower's utility, and $W_C$ depends on the leader's strategy starting from the next step $t = 1$, namely $\theta_x \circ P$, in the following way:
\begin{equation*}
W_C = W_C(x, \theta_{x}\circ P) := \sup_{Q_C\in \cR}Q_{C,0}(x) f_2(x) + (1-Q_{C,0}(x) )\delta_2\bE_x\left[ J_2 \left(X_1, {\theta_x \circ P},({Q^*_S, \theta_x \circ Q_C})\right)\right].
\end{equation*}
For the remainder of this subsection, we omit the superscript $Q^*_S$ and denote $Q_C$ by $Q$ whenever no confusion arises. With a slight abuse of notation, denote\begin{equation}\label{eq:J2new}
    J_2(x,P,Q): = J_2 \left(x, P,{Q^*_S,  Q_C}\right), \quad \forall x \in \bX,\,\, \forall P, Q \in \cR.
\end{equation}
For each $x \in \bX$, define the set of all  feasible continuation values $W_C(x,P)$ by
\begin{equation*}
D_x := \{W_C(x,P) \mid P \in \cR\}.
\end{equation*}

Before studying the leader's optimization problem, we prove some properties of $W_C$. 
Lemma~\ref{lemma:w_continuous} shows that for each $x \in \bX$, the function $P \mapsto W_C(x,P)$ is continuous. Lemma~\ref{lemma:WV_iteration} derives recursive equations satisfied by $W_C$.  Both results rely on the following assumption and is used in the proof of Lemma~\ref{lemma:D_x}, which shows that for each $x \in \bX$, the set $D_x$ is a closed interval, and its endpoints can be characterized using the DPP. 

\begin{assumption}\label{asp:bounded_pi}
    The density function $\pi$ is bounded, i.e., \[
   \|\pi\|_{\infty}:= \sup_{x, y \in \bX} |\pi(x,y)| < \infty.
    \]
\end{assumption}

\begin{lemma} \label{lemma:w_continuous}
Let Assumption~\ref{asp:bounded_pi} hold. For any $x\in\bX$, the follower's payoff function $P\mapsto W_C(x, P)$ is continuous under the topology defined by \eqref{eq:distance}.
\end{lemma}



\begin{lemma} \label{lemma:WV_iteration}
Let Assumption~\ref{asp:bounded_pi} hold. The follower's expected payoff function $(x,P) \mapsto W_C(x,P)$ is Borel measurable. For each $P \in \cR$ and $x \in \bX$, the following Bellman equation holds :
\begin{equation} \label{w_bellman}
W_C(x, P) = \max \{f_2(x), \delta_2 \bE_x[W(X_1, P)] \}, \quad \forall x\in \bX,
\end{equation}
where the expectation $ \bE_x[W(X_1, P)] $ admits the following integral representation:
\begin{equation}\label{eq111}
\bE_x[W(X_1, P)] = \int_{\bX} [P_0(y) W_S(y) + (1-P_0(y))W_C(y, \theta_y \circ P)] \Pi(x, dy).
\end{equation}
\end{lemma}

Since each Markov policy $p\in \cR_0$ can be regarded as a degenerate path-dependent policy $P \in \cR$ with $\theta_x \circ P = P = p$ for any $x \in \bX$,  the following corollary follows directly from Lemma~\ref{lemma:WV_iteration}.

\begin{corollary} \label{cor:markov_iteration}
Consider the expected payoff functions when the leader adopts only randomized Markov stopping policies $p\in \cR_0$. We have 
\begin{equation}\label{eq:w_markov_bellman}
W_C(x, p) =  \max\left\{ f_2(x), \delta_2\int_{\bX}  [p_y W_S(y) + (1-p_y)W_C(y, p)] \Pi(x,dy)  \right\}.
\end{equation}  

\end{corollary}


\begin{lemma} \label{lemma:D_x}
For each $x \in \bX$, denote
$$\underline{w}_x := \inf_{P\in \cR} W_C(x, P)\quad\text{and}\quad\overline{w}_x := \sup_{P\in \cR} W_C(x, P).$$
Then the functions $x \mapsto\underline{w}_x$ and $x \mapsto\overline{w}_x$ are Borel measurable and are
the unique solutions to the following Bellman's equations.
\begin{align} 
\label{e2}\underline{w}_x &= \max \left\{f_2(x), \delta_2 \int_{\bX} \left(W_S(y)\wedge \underline{w}_y \right)\Pi(x,dy)  \right\}, \quad \forall x \in \bX,\\ 
\label{e3}\overline{w}_x &= \max  \left\{f_2(x), \delta_2\int_{\bX}\left(W_S(y)\vee \overline{w}_y \right) \Pi(x,dy)  \right\}, \quad \forall x \in \bX.
\end{align}
Moreover,
\[D_x = [\underline{w}_x, \overline{w}_x].\]
\end{lemma}

Define
$$D:=\{(x,w)\in\X\times\mathbb{R}:\ x\in\X,w\in D_x\}.$$
From the above lemma we know that $D$ is Borel measurable and $D_x$ is closed for each $x\in\X$. Then by the Kuratowski-Ryll-Nardzewski measurable selection theorem \cite[Theorem 4.1]{doi:10.1137/0315056}, there exists a Borel measurable function $w(\cdot):\X\mapsto\mathbb{R}$ such that $w(x)\in D_x$ for any $x\in\X$. Denote the set of such Borel measurable selectors by $D^{\bX}$, i.e.,
\[
    D^{\bX}: = \{w: (\bX, \cB(\bX)) \to (\bR, \cB(\bR)) \mid w(x) \in D_x, \,\, \forall x \in \bX\}.
    \]

\subsubsection{The leader's problem}
From the preceding section,  the follower adopts the following optimal response functions\[
Q^*_S(x) = \one_{\{h_2(x)\ge g_2(x)\}}, \quad Q^*_C(x,P) =  \one_{\{f_2(x)=W_C(x,P)\}},\,\, \forall x\in \bX, P \in \cR. 
\]
Given the resulting optimal response function $Q^*(P) = (Q^*_S, Q^*_C(P))$, we define the leader's utility level at stopping policy $P \in \cR$ and initial state $x$ by $V(x,P)$. This utility admits the decomposition
\begin{equation} \label{eq:v_def}
V(x,P) := J_1\left(x, \tau^P, \tau^{Q^*(P)}\right) = P_0(x) V_S(x) + (1-P_0(x)) V_C(x, \theta_x \circ P),
\end{equation}
where 
\begin{equation*}
V_S(x)  = \bE_x\left[F_1(\tau^P, \tau^{Q^*}) \mid \tau^P = 0 \right]
\end{equation*}
denotes the utility conditional on the leader stopping at the current step, and 
\begin{equation*}
V_C(x, \theta_x \circ P) = \bE_x \left[\delta_1^{\tau^P\wedge\tau^{Q^*}}F_1(\tau^P, \tau^{Q^*}) \mid \tau^P > 0 \right]
\end{equation*}
denotes the utility conditional on the leader continuing at the current step, and it depends on $P$ only through $\theta_x \circ P$.
%
%
%
%
%

It is straightforward that\[
V_S(x) = h_1(x) \one_{\{h_2(x) \ge g_2(x)\}} + f_1(x) \one_{\{h_2(x) < g_2(x)\}}.
\]
By taking conditional expectation, we obtain the following recursive equation satisfied by the leader's continuation utility $V_C$:  for each $P \in \cR, x \in \bX$,
 \begin{equation*}
 \begin{split}
      V_C(x,P) 
      =&g_1(x) \one_{\{W_C(x,P) = f_2(x)\}} + \delta_1\bE_x[V(X_1,P)] \one_{\{W_C(x,P) > f_2(x)\}}.
 \end{split}
 \end{equation*}

By definition, the utility function $V(x,P)$ has the following direct decomposition:
    \begin{align*}
        V(x,P) 
         =& \sum_{t = 0}^{\infty}  \delta_1^t \bE_x\left[\prod_{k =0}^{t-1} \left( (1-P_k)(1- Q^*_k \right)\left[V_S(X_t) P_t + g_1(X_t) (1-P_t) Q^*_t\right]\right],
    \end{align*}
where $Q^*_k = Q^*_C(x_k, \theta_{x_0, \cdots, x_{k-1}}\circ P) = \one_{\{f_2(x_k) = W_C(x_k, \theta_{x_0, \cdots, x_{k-1}}\circ P)\}}, \forall k \in \bN_0$ is Borel measurable. 

It is easy to see that when $P \mapsto Q^*(P)$ is continuous, the map $P \mapsto V(x,P)$ is continuous, and hence the map $(x,P) \mapsto V(x,P)$ is Borel measurable. By a monotone class argument, the same mapping is Borel measurable if $P \mapsto Q^*(P)$ is Borel measurable.  Consequently, $(x,P) \mapsto V_C(x,P)$ is also Borel measurable, and 
the expectation $\bE_x[V(X_1,P)]$ admits the following integral representation:
 \begin{equation*}
 \bE_x[V(X_1,P)] = \int_{\bX}[P_0(y) V_S(y) + (1-P_0(y)) V_C(y, \theta_y \circ P)] \Pi(x, dy).
 \end{equation*}
Similarly, for Markov stopping policy $p \in \cR_0$, the continuation value $V(x,p)$ satisfies recursive equation

\begin{equation}\label{eq:v_markov_bellman}
   V_C(x,p) = g_1(x) \one_{\{W_C(x,p) = f_2(x)\}} + \delta_1 \left(\int_{\bX}[p_y V_S(y) + (1-p_y)V_C(y, p)]\Pi(x,dy) \right) \one_{\{W_C(x,p) > f_2(x)\}}.
\end{equation}

\begin{definition} \label{def:infinite_precommit}
We call a randomized stopping policy $P^* \in \cR$ a precommitment strategy if it is optimal in the sense that 
\begin{equation} \label{eq:infinite_precommit}
	V(x, P^*) = \sup_{P\in \cR} V(x, P), \quad \forall x \in \bX,
\end{equation}
where $V(x, P)$ is defined by \eqref{eq:v_def}.

\end{definition}

To study the precommitment strategy, we draw inspiration from \cite{MR4557638,MR909450}. Thanks to equation~\eqref{eq:v_def}, the leader's optimal control problem~\eqref{eq:infinite_precommit} can be simplified as
\begin{equation*}
\sup_{P \in \cR} V(x,P) = \max\left\{V_S(x),  \sup_{P\in \cR}V_C(x, P)\right\}.
\end{equation*}
In the remainder of this section, we focus on characterizing $\sup_{P\in \cR}V_C(x, P)$. The key idea is to view the leader’s continuation value $V_C$ as a function of the follower’s continuation value $W_C$. The entire infinite sequence $P$ is effectively summarized by a scalar value $W_C(x,P)$ at each state $x$, capturing the follower’s continuation value. This perspective allows us to formulate a Bellman equation on the space of all feasible values of  $W_C(x, P)$, leading to a tractable characterization of the leader’s utility function  $v(w)$ as given in Proposition~\ref{thm:v_bellman}.

\begin{remark}\label{rmk:explain_inconsistency}
For any $P \in \cR$, the continuation value $V_C(\cdot,P)$ satisfies the iteration equation
\[
V_C(x,P)=\one_{\{W_C(x,P)=f_2(x)\}} g_1(x)+\one_{\{W_C(x,P)>f_2(x)\}}\delta_1\int_{\bX}\bigl[P_0(y)V_S(y)+(1-P_0(y))V_C(y,\theta_y\circ P)\bigr]\Pi(x,dy).
\]
In a classical time-consistent problem, one would expect to derive a DPP directly from such an iteration equation. For instance, defining
\[
V_C^*(x):=\sup_{P\in\cR}V_C(x,P),
\]
one would formally hope for a relation of the form
\begin{align*}
V_C^*(x)=& \sup_{\substack{P_0\in\cR_0, \\ \theta_y \circ P \in \cR,  \forall y \in \bX}}\one_{\{W_C(x,P)=f_2(x)\}} g_1(x)\\
&+\one_{\{W_C(x,P)>f_2(x)\}}\delta_1\int_{\bX}\bigl[P_0(y)V_S(y)+(1-P_0(y))V_C(y,\theta_y\circ P)\bigr]\Pi(x,dy)\\
=&\one_{\{W_C(x,P)=f_2(x)\}} g_1(x)+\one_{\{W_C(x,P)>f_2(x)\}}\sup_{P_0\in\cR_0} \delta_1 \int_{\bX}\bigl[P_0(y)V_S(y)+(1-P_0(y))V_C^*(y)\bigr]\Pi(x,dy).
\end{align*}
However, the continuation value $V_C^*$ involves the follower's optimal response, namely, the indicator terms
$\one_{\{W_C(x,P)=f_2(x)\}}$
and 
$\one_{\{W_C(x,P)>f_2(x)\}}$,
which depend on the entire path-dependent strategy $P$, rather than only on the current action $P_0$ and future optimal values $V^*_C$. Consequently, the optimization over $P$ cannot be decoupled into a current optimization over $P_0$ and a continuation optimization over the future strategy $\theta_y\circ P$. This prevents the optimization problem from being reduced to a standard Bellman recursion in terms of the current state alone, and illustrates the intrinsic time inconsistency in the Stackelberg stopping game.
\end{remark}


For each $x \in \bX, w \in D_x$, we define the leader's utility function $v(x,w)$ by 
\begin{equation*}
v(x,w) := \sup_{P\in \cL(x,w)} V_C(x,P),
\end{equation*}
where the set $\cL(x,w)$ is given by
\begin{equation*}
\cL(x,w) := \{P \in \cR \mid W_C(x, P) = w\}.
\end{equation*}
Since $P \to W_C(x, P)$ is continuous, for each feasible $w \in D_x$, the set $\cL(x,w) \subset \cR$ is non-empty and closed. 
%
%
Intuitively, $v(x,w)$ represents the maximum expected utility that the leader can achieve at state $x$, under the constraint that the follower’s continuation value is exactly $w$. It is easy to check that the optimization problem $\sup_{P \in \cR} V_C(x,P)$ can be reduced to 
$\sup_{w \in D_x} v(x,w)$,
as indicated by the following result.
\begin{lemma}
For each $x \in \bX$, we have 
\begin{equation*}
\sup_{w \in D_x} v(x,w) = \sup_{P \in \cR} V_C(x,P).
\end{equation*}
\end{lemma}

\begin{proof}
For each fixed $w \in D_x$, we have $\cL(x,w) \subset \cR$ and
\begin{equation*}
v(x,w) = \sup_{P \in \cL(x,w)} V_C(x,P) \le \sup_{P \in \cR} V_C(x,P).
\end{equation*}
Thus $\sup_{w \in D_x} v(x,w) \le \sup_{P \in \cR} V_C(x,P)$.

For each fixed $P \in \cR$, let $w = W_C(x,P)$. Then $w \in D_x, P \in \cL(x,w)$ and 
\begin{equation*}
V_C(x,P) \le v(x,w) \le \sup_{w \in D_x} v(x,w).
\end{equation*}
Thus $\sup_{P \in \cR} V_C(x,P) \le \sup_{w \in D_x} v(x,w)$.
\end{proof}

In the following theorem, we show that $v(x,w)$ can be characterized by a Bellman's equation~\eqref{eq:v_bellman}.

\begin{theorem} \label{thm:v_bellman}
Let Assumption~\ref{asp:bounded_pi} hold. Then $(x,w)\mapsto v(x,w):D\mapsto\mathbb{R}$ is bounded and upper-semianalytic, and is the unique solution to the Bellman equation
\begin{equation}\label{eq:v_bellman}
v(x,w) = g_1(x)\cdot \one_{\{w = f_2(x)\}}+
\sup_{({w}'_{\cdot},{p})\in \cA(x,w)} \delta_1 \int_{\bX} \left[p_y V_S(y) + (1-p_y) v(y,w'_y)\right] \Pi(x,dy)\cdot \one_{\{w > f_2(x)\}},
\end{equation}
where the admissible set $\mathcal{A}(x,w)$ is defined as
\begin{equation*}
\cA(x,w): = \left\{({w}'_{\cdot},{p}) \in D^{\bX} \times \cR_0 \mid w = \Theta_x({w}'_{\cdot},{p}) \right\},
\end{equation*}
with 
$
\Theta_x({w}'_{\cdot}, {p}):= \delta_2 \int_{\bX} \left[p_y W_S(y) + (1-p_y)w'_y \right] \Pi(x,dy).
$
\end{theorem}

%
%

\begin{remark}
Here $\Theta_x$ represents the update rule for the follower’s continuation value: it maps the stopping probabilities $p_{\cdot}$ and the continuation values ${w}'_{\cdot} \in D^{\bX}$ at the next state to the continuation value ${w}_x$ at the current state $x$. By Lemmas \ref{lemma:WV_iteration} and \ref{lemma:D_x}, $\mathcal{A}(x,w)\neq\emptyset$ for $w\in D_x\cap(f_2(x),\infty)$. 
Equation~\eqref{eq:v_bellman} characterizes the leader’s value function recursively, under the constraint that the follower’s continuation value remains fixed at level $w \in D_x$.
\end{remark}

\begin{remark}
$v(x,w)$ may not be continuous in $w$. A counterexample  is provided by the two-period toy model in Section~\ref{sec:finite_random}, in which by equation \eqref{eq:v_not_continuous} and $w = 4-2p_1$, we have 
\begin{equation*}
v(w) = V_C(0, p_1(w)) = 2\one_{\{w=3\}} + (6-w/2)\one_{\{w>3\}},
\end{equation*}
exhibiting a discontinuity at  $w = 3$. As a consequence, the supremum of $v(x,w)$ over $w\in D_x$, and hence the supremum of $V_C(x, P)$ over $P \in \mathcal{R}$, may not be attained.
\end{remark}

\subsection{Exact Equilibrium Strategy}\label{sec:infinity_equilibrium}
	
In addition to the precommitment strategy in the preceding subsection, a common approach to addressing time inconsistency is to consider an equilibrium strategy, from which the player has no incentive to deviate, assuming all their future selves adopt the same strategy. In our infinite horizon setting, we assume that all future selves adopt a randomized Markov stopping policy $p \in \cR_0$ as defined in Definition~\ref{def:infinite_markov_policy}. 
When restricted to the Markov stopping policies, the infinite-horizon Stackelberg stopping game is characterized by the following objective function pair $(J_1, J_2)$:
\begin{equation} \label{eq:game_exact_equilibrium}
    \begin{split}   
          \text{Follower: } & \sup_{(r, q) \in \cR_0 ^2} J_2(x, p, (r, q));\\
        \text{Leader: }& \,\,\,\sup_{p \in \cR_0 } \quad J_1(x, p, (r^*, q^*)),
    \end{split}
\end{equation}
with \begin{align*}
  &J_i(x, p, (r, q)) := \bE_x\left[ \delta_i^{\tau^{p} \wedge \tau^{(r, q)}} F_i(\tau^{p}, \tau^{(r, q)})\right], \quad i \in \{1,2\}, x\in \bX,\\
  & (r^*, q^*) = \arg\max_{(r, q) \in \cR^2} J_2(x, p, (r, q)),\footnotemark
\end{align*}
\footnotetext{When the argmax is not a singleton, $(r^*,q^*)$ is understood as a Borel measurable selector, such as the earliest optimal stopping strategy.}
where $\delta_1, \delta_2 \in (0,1)$ are the discount factors of the leader and the follower, respectively, and $\tau^p$, $\tau^{r,q}$, and $F_i$ are defined in \eqref{eq:tau_small_p}, \eqref{eq:tau_small_q}, and \eqref{eq:F1F2}, respectively.


\begin{remark}
In the infinite-horizon setting, we restrict attention to Markov equilibrium strategies. This is standard in the literature on time-inconsistent control, where (time-consistent) equilibrium strategies are typically characterized in feedback form as functions of the current state when the model is time-homogeneous and the horizon is infinite; see e.g., \cite{https://doi.org/10.1111/mafi.12293,CL18,CL20,HZ20,MR4397932}. Such a restriction is also economically meaningful. Indeed, because the model is time-homogeneous, at each time the leader faces the same decision problem, provided the game has not yet been stopped. Therefore,
it is natural for the leader to adopt a time-invariant (Markov) strategy. 
\end{remark}

As discussed in subsection \ref{sec:concepts}, the notion of equilibrium strategy in the Stackelberg stopping game corresponds to a feedback Stackelberg equilibrium where the follower's strategy is formulated as a function of the leader's strategy. We revisit these concepts in the infinite horizon setting and discuss their connection and distinction as follows.

\begin{definition}[\textbf{Feedback Equilibrium}]\label{def:feedback_equilibrium_general} 
The follower's Markov strategy is a pair $(r, \textbf{q})${\footnote{For each $p \in \cR_0,x\in \X$, the quantity
$\textbf{q}_x(p)$ represents the conditional probability of stopping given that the current state is $x$, the leader’s
strategy is p, and the leader will continue at current step; and the quantity $r_x$ represents the conditional probability of stopping given that the current state is $x$, and the leader stops at the current step. (In this case, the follower's decision no longer depends on the leader's strategy.)}} $\in \cR_0 \times  \cQ$, where\[
    \cQ :=\{\textbf{q}: \cR_0 \to \cR_0\}.
    \]
 A Markov strategy pair  $(p,(r,\mathbf{q}) )\in \cR_0 \times (\cR_0 \times\cQ)$ is called a randomized feedback equilibrium if the following holds:
    \begin{align*}
        &J_2\left(x,{p'},(r,\mathbf{q}(p'))\right) \ge J_2\left(x, {p'}, (r',\mathbf{q}'(p'))\right), \quad \forall x \in \bX, \, p', r' \in \cR_0, \, \mathbf{q}' \in \cQ;\\
        &J_1\left(x,p,(r,\mathbf{q}(p))\right) \ge J_1\left(x, {p' \oplus_1 p}, (r,\mathbf{q}(p' \oplus_1 p))\right), \quad \forall x \in \bX, \,  p' \in \cR_0,
    \end{align*}
    where ${p}' \oplus_1 {p}$ denotes a deviation from the Markov strategy ${p} \in \cR_0$ representing the strategy of using ${p}'$ at the first step and switching back to ${p}$ for the remaining time. 
\end{definition}


Here the first condition guarantees that $(r, \textbf{q})$ is a best response function for the follower without
assuming that the follower always chooses the earliest optimal stopping time, and the second condition means that the leader has no incentive to deviate from strategy $p$ at the current step, given that her future selves all stick to strategy $p$.
By \eqref{eq:v_def}, the follower’s response $\textbf{q}$ essentially depends only on the leader’s future strategy since the
follower can observe that the leader will continue at the current state. We thus conclude that \[
J_1(x, {p' \oplus_1 p}, (r,\textbf{q}(p' \oplus_1 p))) = J_1(x, p' \oplus_1 p, (r,\textbf{q}(p))), \quad p'\in \cR_0, (r,\textbf{q}) \in \cR_0 \times \cQ.
\]
 Therefore, the above definition can be reduced to the following
equivalent characterization.
\begin{proposition} \label{prop:feedback_char}
Let a pair of randomized Markov stopping policies $(p, (r,q)) \in \cR_0 \times (\cR_0 \times\cR_0))$ satisfy the following:
    
    \begin{align}
        &J_2\left(x,{p' \oplus_1 p},(r,q)\right) \ge J_2\left(x, {p' \oplus_1 p}, (r',q')\right), \quad \forall x \in \bX, \, p', q' , r'\in \cR_0; \label{eq:j2condition} \\
        &J_1\left(x,p,(r,q)\right) \ge J_1\left(x, {p' \oplus_1 p}, (r,q)\right), \quad \forall x \in \bX, \,  p' \in \cR_0. \label{eq:j1condition}
    \end{align}
    Let $(r,\mathbf{q} )\in \cR_0 \times \cQ$ be defined as
    \begin{equation*}
    r_x  \begin{cases}
        =1, & \text{ if }  g_2(x) < h_2(x),\\
        =0, & \text{ if }  g_2(x) > h_2(x),\\
        \in [0,1], & \text{ if }  g_2(x) = h_2(x),
    \end{cases} ; \quad \mathbf{q}(p')  = \begin{cases}
   q, & \text{ if }  p' = p;\\
   \arg\max_{q'  \in [0,1]^N} J_2\left(\cdot, {p'}, (r,q')\right), & \text{ if }  p' \ne p.\\
    \end{cases}
    \end{equation*}
    Then $(p,(r, \mathbf{q})) \in \cR_0 \times (\cR_0 \times\cQ)$ is a randomized feedback equilibrium.
    \end{proposition}
    
Note that a feedback equilibrium does not impose any tie-breaking rule for the follower, which implies that the follower optimal response strategy $(r, \textbf{q})$ may not be unique. However, in the preceding setting, we assume that the follower always chooses the \emph{earliest} stopping time that
attains the maximal expected payoffs. This assumption is reasonable, as choosing the earliest optimal
stopping time avoids additional risk and simplifies the definition of the follower’s best response function as follows.
\begin{definition}[\textbf{Exact Equilibrium}] \label{def:infinite_equilibrium}
    For each $p \in \cR_0$, denote the (earliest) optimal response function of the follower by $(r^*, q^*(p))$. Specifically,
\begin{equation*}
     r^*_x = \one_{\{h_2(x)\ge g_2(x)\}}, \quad q^*_x(p) := \one_{\{f_2(x) \ge W_C(x,p)\}}, 
\end{equation*}
where $W_C(\cdot,p)$ is characterized by the Bellman equation \eqref{eq:w_markov_bellman}.
Denote the leader's value function by \[
V(x,p) = J_1(x, {p},(r^*, q^*(p)))
\]

A randomized Markov stopping policy ${p} \in \cR_0$ is a randomized exact equilibrium strategy{\footnote{In this context, the term `exact' is introduced to distinguish this solution from the `regular' randomized equilibrium strategy obtained via entropy regularization (see Definition \ref{def:regular_equilibrium}). When necessary, we also refer to the feedback equilibrium as the `exact’ feedback equilibrium.}} if 
\begin{equation} \label{eq:exact_def}
V(x, {p}) \ge V(x,  {p'} \oplus_1 {p}), \quad \forall x \in \X, p' \in \cR_0.
\end{equation}
\end{definition}

\begin{remark}\label{rmk:feedback_vs_exact}
By definition, $(r^*, q^*(p))$ satisfies \eqref{eq:j2condition} for all $p \in \cR_0$, and the condition \eqref{eq:exact_def} implies \eqref{eq:j1condition}. Therefore, any exact equilibrium in Definition \ref{def:infinite_equilibrium} also corresponds to a feedback equilibrium in Definition \ref{def:feedback_equilibrium_general}, thanks to Proposition \ref{prop:feedback_char}. However, the converse does not hold. 
\end{remark}

In contrast to the finite-horizon case, we cannot rely on backward induction to construct an equilibrium strategy in the infinite-horizon setting. The following proposition shows that the existence of a randomized exact equilibrium strategy is not guaranteed.

\begin{proposition} \label{prop:ex_inf_horz}
The randomized exact equilibrium strategy defined in Definition~\ref{def:infinite_equilibrium} may fail to exist.
\end{proposition}


Since an exact equilibrium strategy may not exist, and the existence of feedback equilibria is difficulty to establish due to the multiplicity of the follower's best responses w.r.t. the leader's strategy (see Remark \ref{rmk:uncountable}), we turn to the notion of an $\varepsilon$-equilibrium.

\begin{definition}[\textbf{$\varepsilon$-Feedback Equilibrium}]\label{def:epsilon_feedback_equilibrium}
A Markov strategy pair  $(p,(r,\mathbf{q}) )\in \cR_0 \times (\cR_0 \times\cQ)$ is called a randomized $\varepsilon$-feedback equilibrium for problem \eqref{eq:game_exact_equilibrium}, if the following holds:
    \begin{align*}
        &J_2\left(x,{p'},(r,\mathbf{q}(p'))\right) \ge J_2\left(x, {p'}, (r',\mathbf{q}'(p'))\right) - \varepsilon, \quad \forall x \in \bX, \, p', r' \in \cR_0, \, \mathbf{q}' \in \cQ;\\
        &J_1\left(x,p, (r,\mathbf{q}(p))\right) \ge J_1\left(x, {p' \oplus_1 p}, (r,\mathbf{q}(p' \oplus_1 p))\right)-\varepsilon, \quad \forall x \in \bX, \,  p' \in \cR_0.
    \end{align*}
\end{definition}

Similar to the characterization of Definition \ref{def:feedback_equilibrium_general}
in Proposition \ref{prop:feedback_char}, we obtain an $\varepsilon$-feedback equilibrium for free if we have an $\varepsilon$-equilibrium defined as follows.
\begin{definition}[\textbf{$\varepsilon$-Equilibrium}]\label{def:epsilon_equilibrium}
A pair of randomized Markov stopping policies $(p,(r,q)) \in \cR_0 \times (\cR_0 \times \cR_0)$ is called an $\varepsilon$-equilibrium for problem \eqref{eq:game_exact_equilibrium} if  
    \begin{align}
        &J_2\left(x,{p' \oplus_1 p},(r,q)\right) \ge J_2\left(x, {p' \oplus_1 p}, (r',q')\right) -\varepsilon, \quad \forall x \in \bX, \, p' \in \cR_0, \, (r',q') \in \cR_0^2; \label{eq:j2condition_epsilon} \\
        &J_1\left(x,p,(r,q)\right) \ge J_1\left(x, {p' \oplus_1 p}, (r,q)\right) - \varepsilon, \quad \forall x \in \bX, \,  p'\in \cR_0. \label{eq:j1condition_epsilon}
    \end{align}
   
\end{definition}

Comparing conditions \eqref{eq:j2condition_epsilon} and \eqref{eq:j1condition_epsilon} with conditions \eqref{eq:j2condition} and \eqref{eq:j1condition} (which are satisfied by an exact equilibrium, see Remark \ref{rmk:feedback_vs_exact}), we see that the $\varepsilon$-equilibrium from Definition \ref{def:epsilon_equilibrium} can be interpreted as an approximation of the exact (feedback) equilibrium strategy in the original (unregularized) game. In the next subsection, the existence of $\varepsilon$-equilibrium is established via considering an entropy-regularized Stackelberg stopping game, in which an entropy term is added to the follower’s optimization problem.
This regularization yields a continuous best-response function, allowing us to establish the existence of a randomized equilibrium in the entropy-regularized game. 

\subsection{Regular Equilibrium Strategy} \label{sec:regular}

An entropy-regularized Stackelberg game with parameter $\lambda >0$ is characterized by the objective function pair $(J_1,J_2^{\lambda})$:

\begin{equation} \label{eq:game_regular_equilibrium}
    \begin{split}   
          \text{Follower: } & \sup_{(r, q) \in \cR_0 ^2} J^{\lambda}_2(x, p, (r, q));\\
        \text{Leader: }& \,\,\,\sup_{p \in \cR_0 } \quad J_1(x, p, (r^{*,\lambda}, q^{*,\lambda})),
    \end{split}
\end{equation}
with \begin{align*}
  &J^{\lambda}_2(x, p, (r, q)) :=
  \bE_x\left[ \delta_2^{\tau^{p} \wedge \tau^{(r, q)}} F_2(\tau^{p}, \tau^{(r, q)})\right] + \lambda \bE_x\left[
 \sum_{t = 0}^{\tau^{{p}} \wedge \tau^{{q},{r}}} \delta_2^t \cH(q_{X_t} \one_{\{\tau^{{p}} >t \}} + r_{X_t} \one_{\{\tau^{{p}} = t \}} )\right],\\
  &J_1(x, p, (r, q)) := \bE_x\left[ \delta_1^{\tau^{p} \wedge \tau^{(r, q)}} F_1(\tau^{p}, \tau^{(r, q)})\right], \quad  x\in \bX, p\in \cR_0, (r,q) \in \cR_0^2;\\
  & (r^{*,\lambda}, q^{*,\lambda}) := \arg\max_{(r, q) \in \cR^2} J^{\lambda}_2(x, p, (r, q)),
\end{align*}
where $\delta_1, \delta_2 \in (0,1)$ are the discount factors of the leader and the follower, respectively, $\tau^p$, $\tau^{r,q}$, and $F_i$ are defined in \eqref{eq:tau_small_p}, \eqref{eq:tau_small_q}, and \eqref{eq:F1F2}, respectively, and $\cH$ is Shannon's entropy defined by
$$\cH(q)=  -q \log(q)-(1-q) \log(1-q),\quad q \in [0,1].$$

\begin{remark}

    The main motivation for introducing the entropy term is that it plays a crucial analytical role in establishing the existence of a regular equilibrium (see Definition \ref{def:regular_equilibrium}), which serves as a tractable approximation to the exact equilibrium, as shown in Proposition \ref{prop:approximation} and Remark \ref{reapp}.  As demonstrated in Section~\ref{sec:special_case}, this regularization becomes particularly important when the state space is uncountable; see Remark~\ref{rmk:uncountable}.

    In addition to the main motivation, the entropy term is potentially essential for the numerical stability of the iterative algorithm used to compute approximate equilibria. By smoothing the follower's optimization problem, the entropy term ensures that the follower's best response, and consequently the leader's utility, depend continuously on the leader's strategy $p$ (see Corollary \ref{cor:continuous_qstar} and Proposition \ref{prop:continuous_vc}). This regularity may significantly contribute to the stability of numerical computation.

    Beyond its technical role, the entropy term also admits natural behavioral and economic interpretations. For instance, the resulting regularized best-response map \eqref{eq:q_star} has the form of a logit (or quantal) response rule, widely used in economics to model decision makers subject to bounded rationality; see~\cite{MCKELVEY19956}. Alternatively, from a learning perspective, the entropy term can be viewed as encouraging exploration in the exploration–exploitation trade-off in reinforcement learning; see, e.g., \cite{JMLR:v21:19-144}.
\end{remark}

\subsubsection{The follower's problem} For a fixed constant $\lambda > 0$, denote the follower's value function by 
\begin{equation*}
W^{\lambda}(x,{p}): = \sup_{(r,q) \in \cR_0^2} J_2^{\lambda}(x, {p},(r,q) ),
\end{equation*}
which can be decomposed as
\begin{equation*}
W^{\lambda}(x, {p}) = p_x W^{\lambda}_S(x) + (1-p_x) W^{\lambda}_C(x,{p}),
\end{equation*}
with
\begin{equation} \label{eq:w_lambda_s}
W^{\lambda}_S(x) := \sup_{r_x \in [0,1]} r_x h_2(x) + (1-r_x) g_2(x) + \lambda \cH(r_x),
\end{equation}
and 
\begin{equation*}
W^{\lambda}_C(x,{p}) := \sup_{{q}, {r} \in \cR_0} q_x f_2(x)+(1-q_x)\delta_2\bE_x[J^{\lambda}_2(X_1, p, (r,q))] + \lambda \cH(q_x).
\end{equation*}

Thanks to the concavity of $\cH$, we obtain the following result by straightforward computation.
\begin{lemma} \label{lemma:r_star}
There exists a unique $r^{*, \lambda}_x \in [0,1]$ that attains the optimal value in \eqref{eq:w_lambda_s}. It is given by
\begin{equation} \label{eq:r_star}
r^{*, \lambda}_x = \frac{1}{1+ \exp \left(\frac{g_2(x) - h_2(x)}{\lambda}\right)}.
\end{equation}
The utility $W^{\lambda}_S$ is given by 
\begin{equation*}
W^{\lambda}_S(x) = h_2(x) + \lambda \log\left(1+\exp\left(\frac{g_2(x)-h_2(x)}{\lambda}\right)\right).
\end{equation*}
\end{lemma}


Recall that the set of randomized Markov (stationary) stopping policies is defined by
\begin{equation*}
    \cR_0 := \{p: (\bX, \cB(\bX)) \to ([0,1], \cB([0,1])) \}.
\end{equation*}

Under Assumption~\ref{asp:reference_prob}, we have that $\cR_0 \subset L^{\infty}(\bX, \mu)$, which is compact in weak-* topology $\sigma(L^{\infty}(\bX,\mu), L^1(\bX, \mu))$. We say a sequence of function $\{p^n\}_{n\in \bN} \in \cR_0$ converges to $p \in \cR_0$, i.e., \[
p^n \xrightarrow{w*} p, 
\]
if and only if for any test function $\phi \in L^1(\bX, \mu)$,\[
\int_{\bX} p^n(y)\phi(y)\mu(dy) \to \int_{\bX} p(y)\phi(y)\mu(dy).
\]

The following lemma will be used repeatedly in establishing continuity properties.
\begin{lemma} \label{lemma:convergence}
Let $\{p^n\}_{n \in \bN}$ be a sequence of functions in $\cR_0$  and  $p^n \xrightarrow{w*} p \in \cR_0$. Let $\{\phi^n\}_{n \in \bN}$ be a sequence of measurable functions $\phi^n : \bX \to \bR$ with $\sup_{n\in\bN} \sup_{x\in \bX} |\phi^n(x)| < \infty$ and $ \phi^n(x) \to \phi(x)$ for $\mu$-a.e.~$x\in \bX$. Then we have \[
\bE_x[p^n_{X_1} \phi^n(X_1)]   \to \bE_x[p_{X_1} \phi(X_1)], \quad \forall x \in \bX.
\]

\end{lemma}
\begin{proposition}\label{prop:continuous_wc}
    For each $x\in \bX$, the follower's continuation value $W^{\lambda}_C(x,p)$  is continuous w.r.t.~$p$ under the weak-* topology. That is,  $p^n \xrightarrow{w*} p$ implies $W^{\lambda}_C(x,p^n) \to W^{\lambda}_C(x,p)$. 
\end{proposition}

The next corollary follows directly from Proposition~\ref{prop:continuous_wc}.
\begin{corollary}\label{cor:continuous_qstar}
    The follower's response function $q^{*,\lambda}$ is given by
    \begin{equation}\label{eq:q_star}
q^{*,\lambda}_{x} :=q^{*,\lambda}_{x} (p):= \frac{1}{1+ \exp \left(\frac{\delta_2 \bE_x[p_{X_1}W^{\lambda}_S(X_1) + (1-p_{X_1})W^{\lambda}_C(X_1,p)]- f_2(x)}{\lambda}\right)}.
\end{equation} 
Moreover, $p^n \xrightarrow{w*} p$ implies $q^{*,\lambda}_x(p^n) \to q^{*,\lambda}_x(p)$ for each $x \in \bX$.
\end{corollary}

\subsubsection{The leader's problem}Given the follower's optimal response $(r^{*,\lambda}, q^{*,\lambda}(p))$ to the leader's strategy $p$, in equations \eqref{eq:r_star} and \eqref{eq:q_star}, the entropy-regularized game \eqref{eq:game_regular_equilibrium} now can be reduced to the leader's problem: 
\begin{equation} \label{eq:leader_problem_regular}
\sup_{p \in \cR_0}V^{\lambda}(x, p),  \text{ with } V^{\lambda}(x, p) := J_1(x, p, (r^{*,\lambda},q^{*,\lambda}(p)))
\quad \forall x \in \bX, p\in \cR_0.
\end{equation}

Following similar reasoning as in Corollary~\ref{cor:markov_iteration}, the value function $V^{\lambda}$ in \eqref{eq:leader_problem_regular} can be decomposed as
\begin{equation*}
V^{\lambda}(x, p) = p_x V^{\lambda}_S(x) + (1-p_x) V^{\lambda}_C(x, p),
\end{equation*}
where the stopping value $V^{\lambda}_S(x)$ is given explicitly by
\begin{equation}\label{eq:vs_lambda}
V^{\lambda}_S(x) := \bE_x[F_1(\tau^{p}, \tau^{r^{*,\lambda},q^{*,\lambda}}) \mid \tau^{p} = 0] = r^{*,\lambda}_x h_1(x) + (1-r^{*,\lambda}_x) f_1(x);
\end{equation}
and the continuation value $V^{\lambda}_C(x, p)$ satisfies iteration equation
\begin{equation*}
V^{\lambda}_C(x, p) := \bE_x[\delta_1^{\tau^{p} \wedge \tau^{r^{*,\lambda},q^{*,\lambda}}}F_1(\tau^{p}, \tau^{r^{*,\lambda},q^{*,\lambda}}) \mid \tau^{p} > 0] = q^{*,\lambda}_x(p) g_1(x) + (1-q^{*,\lambda}_x(p)) \delta_1 \bE_x[V^{\lambda}(X_1, p)].
\end{equation*}

When the leader deviates from the strategy $p$ to $p'\oplus_1 p$, the follower's response remains unchanged, as it depends only on the leader's current state and future strategy. The corresponding value function of the leader is therefore given by
\begin{equation}\label{eq:v_lambda_decompose}
V^{\lambda}(x,p'\oplus_1 p) = p'_x V^{\lambda}_S(x) + (1-p'_x)V^{\lambda}_C(x, p).
\end{equation}

Analogously to the exact equilibrium defined in Definition~\ref{def:infinite_equilibrium}, we introduce the notion of a regular equilibrium in the entropy-regularized Stackelberg stopping game. 
\begin{definition}[\textbf{Regular Equilibrium}] \label{def:regular_equilibrium} A randomized Markov stopping policy $p \in \cR_0$ is called a randomized regular equilibrium for the entropy-regularized Stackelberg stopping game \eqref{eq:game_regular_equilibrium} with parameter $\lambda > 0$ if, for all $x \in \bX$, 
\begin{equation*}
V^{\lambda}(x,p) \ge V^{\lambda}(x, p'\oplus_1 p), \quad \forall p' \in \cR_0.
\end{equation*}
Equivalently, the strategy pair $(p, (r,q))$ with $(r,q) := (r^{*,\lambda}, q^{*\lambda}(p))$ as defined in \eqref{eq:r_star} and \eqref{eq:q_star}, satisfies that for any $x \in \bX$,
\begin{equation*}
    \begin{split}
        &J_2^{\lambda}(x,p, (r, q)) = \sup_{(r',q')\in \cR_0^2} J_2^{\lambda}(x, p,(r',q')),\\
        &J_1(x, p, (r, q)) = \sup_{p'\in \cR_0} J_1(x, p' \oplus_1 p, (r, q)).
    \end{split}
\end{equation*}
\end{definition}

For any fixed $x$, the value $V^{\lambda}(x, p'\oplus_1 p)$ depends on $p'$ only through the single component $p'_x$; see \eqref{eq:v_lambda_decompose}. This observation leads to the following proposition, which provides a sufficient condition for the above inequality to hold for almost all $x \in \bX$.

\begin{proposition} \label{prop:equilibrium_characterization}
If $p$ is a fixed point of the set-valued map $\Psi: \cR_0 \to 2^{\cR_0}$ defined as follows,
\begin{equation*}
    \Psi(p) := \left\{\tilde p \in \cR_0: \int_{\bX}(1-\tilde p_x )\one_{\{V_S^{\lambda}(x) > V_C^{\lambda}(x,p\}} \mu(dx) = 0  \text{ and } \int_{\bX}\tilde p_x \one_{\{V_S^{\lambda}(x) < V_C^{\lambda}(x,p)\}} \mu(dx)= 0.  \right\},
\end{equation*}
then \[
\mu\left( \left\{x \in \bX :  V^{\lambda}(x,p) \ge V^{\lambda}(x, p' \oplus_1 p), \,\, \forall p' \in \cR_0\right\}\right) = 1.
\]


\end{proposition}

The following proposition is a key step to establish the existence of a regular equilibrium. 
\begin{proposition}\label{prop:continuous_vc}
    For each $x \in \bX$, the leader's continuation value $V^{\lambda}_C(x,p)$  is continuous w.r.t.~$p$ under the weak-* topology. That is, $p^n \xrightarrow{w*} p$ implies $V^{\lambda}_C(x,p^n) \to V^{\lambda}_C(x,p)$. 
\end{proposition}

\begin{remark} \label{rmk:mu_null}
From the proofs of Proposition~\ref{prop:continuous_wc}, Corollary~\ref{cor:continuous_qstar}, and Proposition~\ref{prop:continuous_vc}, the values of $W_C^{\lambda}(\cdot,p)$, $q^*(p)$ and $V_C^{\lambda}(\cdot,p)$ are unchanged if  $p$ is modified on any $\mu$-null set. Indeed, these quantities depend on $p$ only through integrals w.r.t.~$\mu$.
\end{remark}
\begin{theorem}\label{thm}
    For each $\lambda > 0$, there exists a regular randomized equilibrium $p \in \cR_0$ for the entropy-regularized Stackelberg game \eqref{eq:game_regular_equilibrium}.
\end{theorem}

Thanks to Theorem \ref{thm}, although the existence of an exact equilibrium for the original problem \eqref{eq:game_exact_equilibrium} is not guaranteed, the existence of a regular equilibrium ensures that an approximation to the exact equilibrium (i.e., an $\varepsilon$-equilibrium) always exists.

\begin{proposition}\label{prop:approximation}
For any $\varepsilon>0$, there exists $\lambda_0:=\frac{(1-\delta_2)\epsilon}{\log 2}>0$ such that for any $\lambda\in(0,\lambda_0)$, the regular randomized equilibrium w.r.t. $\lambda$ is an $\varepsilon$-equilibrium in the sense of Definition \ref{def:epsilon_equilibrium} for the game \eqref{eq:game_exact_equilibrium}.
\end{proposition}

\begin{remark}\label{reapp}
Note that the existence of $\varepsilon$-equilibrium cannot be obtained by an argument involving discretizing the state space, because we do not assume the payoff functions are continuous in states.
\end{remark}

\subsection{Special Cases and Application Example}\label{sec:special_case}
\subsubsection{Convergence of regular equilibrium in the finite state space}

In this subsection, we consider the special case in which the state space $\bX$ is finite. Without loss of generality, let $\bX = \{1, \cdots ,N\}$ for some $N \in \bN$. 
When the state space is finite, as the entropy parameter $\lambda \to 0$, the family of regular randomized equilibria $\{p^{\lambda}\}_{\lambda>0}$ admits at least one limit point $p^*$. Moreover, any such limit point is an exact feedback equilibrium as defined in  Definition~\ref{def:feedback_equilibrium_general} for the unregularized problem \eqref{eq:game_exact_equilibrium}. To illustrate this, we explicitly construct an exact feedback equilibrium for the counterexample in the proof of Proposition~\ref{prop:ex_inf_horz}. In contrast, these observations do not extend to the case of an uncountable state space. This limitation highlights the necessity of introducing the entropy regularization term in the general setting considered in the previous section. 

Recall that the follower essentially  only need to solve for his optimal response conditional on the leader continuing at the current step. When the leader stops at the current step, the follower's optimal response is independent of the leader's strategy $p$ and can be determined directly. Accordingly, throughout the remainder of this subsection we define the follower's response in this case by
\begin{equation} \label{eq:r_star_new}
r^*_x =
\begin{cases}
    1, & \text{if } g_2(x) < h_2(x),\\[0.3em]
    \frac{1}{2}, & \text{if } g_2(x) = h_2(x),\\[0.3em]
    0, & \text{if } g_2(x) > h_2(x),
\end{cases}
\qquad \forall x \in \bX.
\end{equation}
In contrast to the earliest optimal stopping time \(\one_{\{g_2(x) \le h_2(x)\}}\), the policy \(r^*\) defined in \eqref{eq:r_star_new} is a natural choice when the earliest-stopping restriction is lifted, as it coincides with the limit of \(r^{*,\lambda}\) in \eqref{eq:r_star} as \(\lambda \to 0\), and also satisfies the requirement in Proposition \ref{prop:feedback_char}. Moreover, for any sequence \(\{\lambda_n\}_{n \in \bN}\) satisfying \(\lambda_n \to 0\), we have
\begin{equation} \label{eq:vs_limit}
\lim_{n \to \infty} V_S^{\lambda_n}(x)
= \lim_{n \to \infty}r^{*,\lambda_n}_x h_1(x) + \bigl(1 - r^{*,\lambda_n}_x\bigr) f_1(x)
= r^{*}_x h_1(x) + \bigl(1 - r^{*}_x\bigr) f_1(x)
= V_S(x).
\end{equation}
   
The following proposition shows that the family of regular randomized equilibria $\{(p^{\lambda}, (r^{*, \lambda},q^{*, \lambda})\}_{\lambda>0}$ admits at least one limit point $(p^*, (r^*,q^*))$ that satisfies \eqref{eq:j2condition} and \eqref{eq:j1condition}, and hence prove the existence of exact feedback equilibrium.
    
\begin{proposition}\label{prop:feedback}
Suppose the state space is finite. Let $\{\lambda_n\}_{n\in \bN}$ be any sequence with $\lambda_n \to 0$, and for each $n$ let $p^n$ be a regular randomized equilibrium of the entropy-regularized Stackelberg stopping game with parameter $\lambda_n$. Then there exists a limit point of the sequence
$\left\{ \left(p^n,\; (r^{*, \lambda_n},q^{*,\lambda_n}(p^n))\right)\right\}_{n}$
that constitutes a randomized feedback equilibrium as defined in Definition \ref{def:feedback_equilibrium_general}.
\end{proposition}

\begin{remark}
As an alternative to studying limit points of regular randomized equilibria, the existence of an exact randomized feedback equilibrium in the finite-state space case can be established directly via Kakutani’s fixed-point theorem. More precisely, a pair $(p^*, q^*) \in [0,1]^N \times [0,1]^N$ satisfies \eqref{eq:j2condition} and \eqref{eq:j1condition} if and only if it is a fixed point of the set-valued map $\hat\Psi : [0,1]^N \times [0,1]^N \to 2^{[0,1]^N \times [0,1]^N}$
defined by
\begin{equation*}
\begin{split}
\hat \Psi (p,q):=
\Bigl\{ (\tilde p,\tilde q) \in [0,1]^N \times [0,1]^N \,\Big|\, 
\tilde q \in \arg\max_{q' \in [0,1]^N} J_2(x, \tau^p, \tau^{q'\oplus_1 q,r^*}),
\tilde p \in \arg\max_{p' \in [0,1]^N} J_1(x, \tau^{p' \oplus_1 p}, \tau^{q,r^*}), 
\ \forall x \in \bX
\Bigr\}.
\end{split}
\end{equation*}
As $\bX$ is finite, the maximizers $\tilde q$ (respectively $\tilde p$) of $J_2$ (respectively $J_1$) are attained for any fixed $(p,q)$, and $V_C$ is continuous, Kakutani’s fixed-point theorem applies on the compact, convex set $[0,1]^N \times [0,1]^N$, yielding the existence of such a fixed point.
\end{remark}

\begin{remark}\label{rmk:uncountable}
    The convergence result and existence proof of a feedback equilibrium does not carry over to the general Polish state space. The main difficulty lies in the lack of compactness and continuity properties in infinite-dimensional spaces. In particular, even if $p^n \xrightarrow{w*} p$ and $q^n \xrightarrow{w*} q$, this does not imply convergence of $\bE_x[p^n_{X_1} q^n_{X_1}]$, since weak-* convergence preserves integrals against fixed functions but not products of weak-* convergent sequences. Consequently, the set-valued map $\hat \Psi$ may fail to have a closed graph, preventing a direct application of Kakutani’s fixed-point theorem.
\end{remark}

\begin{remark}
Recall that the example in Proposition~\ref{prop:ex_inf_horz} does not admit an exact equilibrium under Definition~\ref{def:infinite_equilibrium}. However, there always exists an exact feedback equilibrium according to Proposition~\ref{prop:feedback}. Specifically, one can verify that there exists a unique exact randomized feedback equilibrium for this example, and the equilibrium is given by $p= (0,p_b, 0), q = (q_a, 0, 1)$, where $p_b$ and $q_a$ are uniquely determined by 
\begin{align*}
&p_b = \frac{\delta_2(\pi_{ba}K + \pi_{bb}W(b) + \pi_{bc}K^2) - W(b)}{\delta_2(\pi_{ba}K + \pi_{bb}W(b) + \pi_{bc}K^2) - 2} \in (0,1), \text{ where }
 W(b) = \frac{K (1-\delta_2 \pi_{aa})}{\delta_2 \pi_{ab}} >2;\\
&q_a = \frac{V(a) - \delta_1 (\pi_{aa}V(a) + \pi_{ab}K)}{K^2 - \delta_1 (\pi_{aa}V(a) + \pi_{ab}K)} \in (0,1), \text{ where }
V(a) = \frac{K - \delta_1 \pi_{bb}K - 2\delta_1 \pi_{bc}}{\delta_1 \pi_{ba}} \le K^2.
\end{align*} 
\end{remark}
\subsubsection{Sufficient conditions for the existence of exact equilibrium}
The following theorem shows that, under suitable conditions, this Stackelberg stopping game admits an exact equilibrium.

\begin{theorem} \label{thm:strict_condition}
    Denote by $A(p)$ the indifference set of the follower corresponding to the leader’s strategy $p$, namely,
\[
A(p): = \left\{ x\in \bX: f_2(x) = \delta_2 \bE_x\left[ p_{X_1} W_S(X_1)+ (1-p_{X_1}) W_C(X_1, p)\right]\right\},
\]
where $W_S$ and $W_C$ are defined in \eqref{eq:W_S} and \eqref{eq:w_markov_bellman}, respectively.

If $\mu(A(p)) = 0$ for all $p \in \cR_0$, then there exists an exact equilibrium.
\end{theorem}

The following example satisfies the condition in the theorem above.

\begin{proposition} \label{cor:example}
 Let $\X=[0,1]$ and $\mu$ be the Lebesgue measure. Suppose that the trasition kernal is
\[
\Pi(x, dy)=[1+\eta(2x-1)(2y-1)]dy,
\]
where $\eta \in [0,1]$. Assume that $f_2(x)=x, g_2(x) = 1, h_2(x) = 0$ and $\delta_2 \in (0,1)$.
    Then \[
    \mu(A(p)) = 0, \quad \forall p \in \cR_0.
    \]
    Consequently, there exists an exact equilibrium.
\end{proposition}
\subsubsection{Application to game option with a dominant player }\label{sec:application} Let us consider a Stackelberg variant of the game option example in \cite{HZ09}. In the classical setting of game options, not only does the holder have the right to exercise the option at a chosen time, but the issuer is also allowed to recall or cancel the contract. In the Stackelberg variant, we assume that the issuer is the dominant player and announces her strategy in advance.

Let $X_t$ denote the price of the underlying asset. Suppose that the issuer stops at time $\tau$ and the holder stops at time $\rho$. The payoff received by the holder from the issuer is given by \[
\Gamma(\tau, \rho) := L(X_{\rho}) \one_{\rho<\tau} + U(X_{\tau}) \one_{\rho>\tau} + M(X_{\tau}) \one_{\tau = \rho}.
\]   
Here, $L(X_{\rho})$ (resp. $U(X_{\tau})$, $M(X_{\tau})$) denotes the amount received by the holder when he stops before (resp. after, simultaneously with) the issuer.


As discussed in \cite{HZ09}, we introduce utility functions $\varphi_1$ and $\varphi_2$ for the issuer and the holder, respectively. The Stackelberg stopping game can then be formulated as problem \eqref{eq:game_exact_equilibrium}, with payoff functions given by
\begin{equation*}
    \begin{split}
        &f_1(x) = \varphi_1(-U(x)), \quad g_1(x) = \varphi_1(-L(x)), \quad h_1(x) = \varphi_1(-M(x));\\
        &f_2(x) = \varphi_2(L(x)), \quad g_2(x) = \varphi_2(U(x)), \quad h_2(x) = \varphi_2(M(x)).
    \end{split}
\end{equation*}

\section{Proof of Main Results} \label{sec:proof}

\subsection{Proofs in Section \ref{sec:infinity_precommit}}

\begin{proof}[Proof of Lemma \ref{lemma:w_continuous}]

For any response policy $Q\in \cR$, we denote $\tau^Q := \tau^{Q^*_S, Q}$ for simplicy. For each \emph{future} strategy $P \in \cR$, the leader's stopping time is denoted by \[
\sigma^P := \inf\{t\ge 1: P_{t-1}(X_1, \cdots, X_t) \ge Y_t\}.
\]
Then by definition,\[
W_C(x, P) = \sup_{Q \in \cR} \bE_x\left[ \delta_2^{\sigma^P \wedge \tau^Q} F_2(\sigma^P,\tau^Q)\right].
\]
Fix $x\in\bX$ and take $P^n,P\in\mathcal{R}$ with $P^n\to P$. Let $\epsilon>0$. Recall $K$ given in Assumption~\ref{asp:bounded_fgh}. Choose $M\in\bN$ such that  $\delta_2^MK\leq\epsilon$. 
We have that
\begin{align*}
&\left|\bE_x\left[ \delta_2^{\sigma^{P^n} \wedge \tau^Q} F_2(\sigma^{P^n},\tau^Q)\right]-\bE_x\left[ \delta_2^{\sigma^P \wedge \tau^Q} F_2(\sigma^P,\tau^Q)\right]\right|\\
&=\bE_x\bigg|\delta_2^{\sigma^{P^n}\wedge\tau^Q}F_2(\sigma^{P^n},\tau^Q)\left(\one_{\{1 \leq \sigma^{P^n}\leq M\}}+\one_{\{\sigma^{P^n}> M, 0 \leq \tau^Q\leq M\}}+\one_{\{\sigma^{P^n}> M,\tau^Q> M\}}\right)\\
&\quad\quad-\delta_2^{\sigma^{P}\wedge\tau^Q}F_2(\sigma^{P},\tau^Q)\left(\one_{\{1 \leq \sigma^{P}\leq M\}}+\one_{\{\sigma^{P}> M, 0\leq \tau^Q\leq M\}}+\one_{\{\sigma^{P}> M,\tau^Q> M\}}\right)\bigg|\\
&\leq\bE_x\left|\sum_{s=1}^{M}\delta_2^{s\wedge\tau^Q}F_2(s,\tau^Q)\left(\one_{\{\sigma^{P^n}=s\}}-\one_{\{\sigma^{P}=s\}}\right)+\delta_2^{\tau^Q} f_2(\tau^Q)\left(\one_{\{\sigma^{P^n}> M\}}-\one_{\{\sigma^{P}> M\}}\right)\one_{\{0\leq \tau^Q\leq M\}}\right|+2\epsilon\\
&\leq K\bE_x\left[\sum_{s=1}^{M}\left|\one_{\{\sigma^{P^n}=s\}}-\one_{\{\sigma^{P}=s\}}\right|+\left|\one_{\{\sigma^{P^n}> M\}}-\one_{\{\sigma^{P}> M\}}\right|\right]+2\epsilon.
\end{align*}
Since the bound is uniform for all $Q \in \cR$, we have that \[
\left| W_C(x, P^n)- W_C(x, P) \right| \le K\bE_x\left[\sum_{s=1}^{M}\left|\one_{\{\sigma^{P^n}=s\}}-\one_{\{\sigma^{P}=s\}}\right|+\left|\one_{\{\sigma^{P^n}> M\}}-\one_{\{\sigma^{P}> M\}}\right|\right]+2\epsilon.
\]

By induction, one can show that if $P^n \to P$, then for any $n \in \bN_0, k \in \bN$ and $i_1, \cdots, i_k \in \bN_0$, \[
\bE_x \left|\prod_{m = 1}^kP_{i_m}^n(X_1, \cdots, X_{i_m+1}) - \prod_{m = 1}^k P_{i_m}(X_1, \cdots, X_{i_m+1}) \right| \le \sum_{m = 1}^k \bE_x  \left|P_{i_m}^n(X_1, \cdots, X_{i_m+1}) -  P_{i_m}(X_1, \cdots, X_{i_m+1}) \right|.
\] 
For each $ i_m \in \bN_0$,\begin{align*}
&\bE_x  \left|P_{i_m}^n(X_1, \cdots, X_{i_m+1}) -  P_{i_m}(X_1, \cdots, X_{i_m+1}) \right|\\
 \le & \|\pi\|_{\infty} \int_{\bX^{i_m+1}}  \left|P^n_{i_m}(y_1, \cdots, y_{i_m+1}) -  P_{i_m}(y_1, \cdots, y_{i_m+1}) \right|\mu(d y_1) \cdots \mu(d y_{i_m+1}) \to 0, \quad n \to \infty.
\end{align*}
Then for each $1 \le s\le M$, 
\begin{align*}
&\bE_x \left|\one_{\{\sigma^{P^n}=s\}}-\one_{\{\sigma^{P}=s\}}\right| \\
=& \bE_x \left| P^n_{s-1}(X_1, \cdots, X_{s})\prod_{k = 0}^{s-2} \left(1-P^n_k(X_1, \cdots X_{k+1})\right) -  P_{s-1}(X_1, \cdots, X_{s})\prod_{k = 0}^{s-2} \left(1-P_k(X_1, \cdots X_{k+1})\right)\right| \to 0,\quad n\to\infty.
\end{align*}

Similarly,\[
\bE_x \left|\one_{\{\sigma^{P^n}> M\}}-\one_{\{\sigma^{P}> M\}}\right| \to 0,\quad n\to\infty.
\]
By the arbitrariness of $\epsilon$, the continuity result holds. 
\end{proof}

\begin{proof}[Proof of Lemma \ref{lemma:WV_iteration}]
First, by a similar argument as in the proof of Lemma~\ref{lemma:w_continuous}, we have that for each fixed $x \in \bX$, the map $(P,Q) \mapsto \overline{J_2}(x, P,Q)$ is continuous, where\[
\overline{J_2}(x, P,Q) = \bE_x\left[J_2(X_1, P,Q) \right], \quad \forall P, Q \in \cR.
\]
Moreover, for fixed $P, Q$, the map $x \mapsto \overline{J_2}(x, P,Q)$ is Borel measurable, by integration against the transition kernel. Thus the function $\overline{J_2}$ is jointly measurable. Let $\{R^n\}_{n \in \bN}\subset \cR$ be a countable dense family. Denote\[
\overline W(x,P) := \sup_{Q \in \cR} \overline{J_2}(x,P,Q).
\] 
Then \[
\overline W(x,P) = \sup_{n\in \bN} \overline{J_2}(x,P,R^n)
\]
is also Borel measurable.
    By definition, for a fixed $P \in \bR$ and $x \in \bX$,
\begin{align*}
W_C(x,P) = & \sup_{Q \in \cR} Q_{0,x} f_2(x) + (1-Q_{0,x}) \delta_2 \bE_x[J_2(X_1, P, \theta_x \circ Q)]\\
= & \sup_{q \in [0,1]]} q f_2(x) + (1-q) \delta_2  \sup_{Q \in \cR}\bE_x[J_2(X_1, P, Q)]\\
= & \max\left\{ f_2(x),  \delta_2  \overline W(x,P) \right\}.
\end{align*}
Therefore, $(x,P) \mapsto W_C(x,P)$ is also Borel measurable. To prove the Bellman equation, it remains to show that for fixed $P \in \cR$, $x \mapsto W(x,P)$ is integrable, and \[
\overline W(x,P) := \sup_{Q \in \cR} \overline{J_2}(x,P,Q)= \bE_x\left[\sup_{Q \in \cR} J_2(X_1,P,Q)\right] = \bE_x\left[W(X_1, P)\right]
\]

For fixed $P \in \cR$, the direct decompostion \[
W(x,P) = P_0(x) W_S(x) + (1-P_0(x)) W_C(x, \theta_x \circ P),
\]
and the Borel map $x \mapsto \theta_x \circ P$ imply that $x \mapsto W(x,P)$ is Borel measurable.

The pointwise upper bound gives \[
\overline W(x,P) \le \bE_x\left[ \sup_{Q\in \cR} J_2(X_1, P,Q) \right] = \bE_x\left[ W(X_1,P) \right].
\]
Now it suffices to show that for any $\varepsilon >0$, there exists $Q^{\epsilon} \in \cR$ such that \[
\overline W(x,P) := \sup_{Q\in\cR} \bE_x\left[ J_2(X_1, P,Q)\right] \ge \bE_x\left[ J_2(X_1, P,Q^{\varepsilon})\right] \ge \bE_x\left[ W(X_1,P) \right] -\varepsilon.
\]

By definition, for each $x \in \bX, P\in \cR$,\[
W(x,P) := \sup_{Q \in \cR} J_2(x, P, Q) = \sup_{q\in [0,1], Q \in \cR} P_0(x) W_S(x) + (1-P_0(x))\left[q f_2(x) + (1-q) \delta \overline{J_2}(x, \theta_x \circ P, Q)\right].
\] 
Since $Q \mapsto \overline{J_2}(x, \theta_x \circ P, Q)$ is continous, we have \[
W(x,P) = \sup_{n \in \bN} J_2(x,P,Q^n),
\]
where the sequence $\{Q^n\}_{n \in \bN} \subset \cR$ is defined by \[
Q^{n}_t(x_0, \cdots, x_t) = \begin{cases} 1, \quad & n = 0, \forall t \ge 0 ;\\
    0,  \quad & t =0, \forall n \ge 1;\\
    R^n_{t-1}(x_1, \cdots, x_t), \quad &\forall t \ge 1, \forall n \ge 1.
\end{cases}
\]
Then for fixed $P \in \cR$, and for any $\varepsilon>0$, define
\[
m_{\varepsilon}(x) := \min\{n \in \bN: J_2(x,P,Q^n) > W(x,P) - \varepsilon\},
\]
which is Borel measurable. Construct $Q^{\varepsilon} \in \cR$ by \[
Q^{\varepsilon}_t(x_0, \cdots, x_t) = Q^{m_{\varepsilon}(x_0)}_t(x_0, \cdots, x_t),
\]
then we have 
\[
J_2(x, P, Q^{\varepsilon})  \ge W(x,P) - \varepsilon,
\]
uniformly for all $x \in \bX$. Therefore the integration leads to \[
\bE_x\left[J_2(X_1, P, Q^{\varepsilon}) \right]  \ge \bE_x\left[W(X_1, P)\right] - \varepsilon
\]

\end{proof}

\begin{proof}[Proof of Lemma \ref{lemma:D_x}]
Since $W_C$ is jointy measurable and $P \mapsto W_C(x,P)$ is continuous, $\overline w$ and $\underline w$ are also Borel measurable. By Lemma \ref{lemma:WV_iteration}, $W_C(x, P)$ satisfies
\begin{equation*}
W_C(x,P) = \max \left\{ f_2(x), \delta_2 \bE_x[P_0(X_1)W_S(X_1) + (1-P_0(X_1))W_C(X_1, \theta_{X_1} \circ P)] \right\}.
\end{equation*}
Taking the infimum over all $P \in \cR$ on both sides yields 
\begin{align*}
    \underline{w}_x =&  \inf_{P \in \cR} \max \left\{ f_2(x), \delta_2 \bE_x[P_0(X_1)W_S(X_1) + (1-P_0(X_1))W_C(X_1, \theta_{X_1} \circ P)] \right\}\\
    = & \max  \left\{ f_2(x), \delta_2  \inf_{P_0 \in \cR_0, \mathbf{P}\in \cR^{\bX}}\bE_x[P_0(X_1)W_S(X_1) + (1-P_0(X_1))W_C(X_1, \mathbf{P}(X_1)] \right\}.
\end{align*}
Using a measurable selection argument, we obtain
\begin{align*}
&\inf_{P_0 \in \cR_0, \mathbf{P}\in \cR^{\bX}}\bE_x[P_0(X_1)W_S(X_1) + (1-P_0(X_1))W_C(X_1, \mathbf{P}(X_1)] \\
=& \bE_x\left[ \inf_{P_0 \in \cR_0}P_0(X_1)W_S(X_1) + (1-P_0(X_1)) \inf_{\mathbf{P}\in \cR^{\bX} }W_C(X_1, \mathbf{P}(X_1)\right]\\
=& \bE_x\left[ \inf_{P_0 \in \cR_0}P_0(X_1)W_S(X_1) + (1-P_0(X_1))  \underline{w}_{X_1}\right]\\
=& \bE_x\left[ \min\left\{W_S(X_1),  \underline{w}_{X_1}\right\}\right].
\end{align*}
In the last equation, the infimum is attained at 
\begin{equation}\label{eq113}
P_0^*(x) = \underline p_x := \one_{\{W_S(x) \le \underline{w}_{x}\}}, \quad x \in \bX.
\end{equation}
Thus we obtain equation \eqref{e2}.  Similarly, $\overline{w}_x$ satisfies \eqref{e3}.

To show \eqref{e2} admits a unique solution, we consider the following operator $\underline{\Phi}: \bR^{\bX} \to \bR^{\bX}$, where $\bR^{\bX}$ denotes the space of all real-valued functions on $\bX$, and is endowed with the supremum norm $\| {w}\|= \sup_{x\in \bX} |w_x|$.
\begin{equation*}
\underline{\Phi}[{w}]_x :=  \max\left\{f_2(x), \delta_2 \int_{\bX} \left(W_S(y)\wedge w_y\right) \Pi(x,dy)\right\}, \quad \forall x \in \bX.
\end{equation*}
Let $\tilde{{w}}$ and $\hat{{w}}$ be two arbitrary functions in $\bR^{\bX}$. We have that

\begin{equation*}
\left|\underline{\Phi}[\hat{w}]_x  - \underline{\Phi}[\tilde{w}]_x\right| \le \delta_2\int_{\bX} |\hat{w}_y - \tilde{w}_y| \Pi(x, dy) \le \delta_2 \|\tilde{w}-\hat{w}\|.
\end{equation*}
Therefore, $\|\underline{\Phi}[\tilde{w}] - \underline{\Phi}[\hat{w}]\| \le \delta_2\|\tilde{w}-\hat{w}\|$. 
By the Banach fixed-point theorem, \eqref{e2} admits a unique solution. Similarly we can show that \eqref{e3} has a unique solution.

To show that $D_x = [\underline{w}_x, \overline{w}_x]$, we first show that the supremum and infimum in the definition of $\overline w_x$ and $\underline w_x$ are attained, and  then we show that for any $w\in[\underline w_x,\overline w_x]$, there exists $P\in\mathcal{R}$ such that $W_C(x,P)=w$. 

Indeed, the optimal values for $\underline{w}_x$ and $\overline{w}_x$ are attained by pure Markov strategies. To see this, consider the pure Markov strategy $\underline p$ defined in \eqref{eq113}.  The above argument shows that $\underline w$ is a solution to the following equation,\[
\underline w_x = \max \left\{ f_2(x), \delta_2 \bE_x\left[\underline p_{X_1}W_S(X_1) + (1-\underline p_{X_1})\underline w_{X_1}\right] \right\} , \quad \forall x \in \bX.
\]By Lemma \ref{lemma:WV_iteration}  , the utility function $W_C(\cdot, \underline p)$  satisfies the same equation.  A similar argument shows that this equation admits a unique solution. Therefore,
\[
W_C(x, \underline{p}) = \inf_{p \in \cR_0}W_C(x, p) = \underline{w}_x, \quad \forall x \in \bX.
\]
The same reasoning applies to $\overline w_\cdot$. 

Fix $w\in[\underline w_x,\overline w_x]$. Let $\overline P,\underline P\in\mathcal{R}$ be such that
$$W_C(x,\overline P)=\overline w_x\quad\text{and}\quad W_C(x,\underline P)=\underline w_x.$$
For $\eta\in[0,1]$, define $P^\eta\in\mathcal{R}$ by
$$P_t^\eta:=\eta\underline P_t+(1-\eta)\overline P_t,\quad\forall\,t\in\bN_0.$$
Then $\eta\mapsto P^\eta$ is continuous. By Lemma~\ref{lemma:w_continuous}, there exists $\eta^*\in[0,1]$ such that $W_C(x,P^{\eta^*})=w$.
\end{proof}

\begin{proof}[Proof of Theorem \ref{thm:v_bellman}]
The boundedness of $v$ follows immediately from Assumption~\ref{asp:bounded_fgh}. As the map $(x,P)\mapsto V(x,P)$ is Borel measurable and the set 
$$\text{Graph}(\cL):=\{(x,w,P)\in D\times\cR\,|\,W_C(x,P)=w\}$$
is Borel measurable, we know that $v$ is upper-semianalytic.

We first show that $v(x,w)$ satisfies \eqref{eq:v_bellman}. For each $P \in \cR$ such that $W_C(x,P) = f_2(x)$, we have $V_C(x, P) =  g_1(x)$. By definition, $v_x(f_2(x)) = g_1(x)$. For the rest of the proof, we assume $w\in D_x\cap(f_2(x),\infty)$. We now establish the two inequalities that together form the Bellman equation \eqref{eq:v_bellman}.

``$\le$". Fix $x$ and $w \in D_x \cap (f_2(x), \infty)$. For each $P \in \cL(x,w)$, let $\hat{w}_y = W_C(y, \theta_y \circ P)$ for $y \in \bX$. Then $(\hat{{w}}_{\cdot}, P_0(\cdot)) \in \cA(x,w)$ and  
\begin{equation*}
\begin{split}
V_C(x,P) &=\delta_1 \int_{\bX} [P_0(y) V_S(y) +(1-P_0(y)) V_C(y, \theta_y \circ P)] \Pi(x,dy) \\ 
& \le  \delta_1 \int_{\bX} [P_0(y) V_S(y) +(1-P_0(y)) v(y,\hat{w}_y)]  \Pi(x,dy)\\ 
& \le  \sup_{({w}'_{\cdot},{p})\in \cA(x,w)} \delta_1 \int_{\bX} [p_y V_S(y) + (1-p_y) v(y,w'_y)] \Pi(x,dy).
\end{split}
\end{equation*}
Since the above holds for all $P \in \cL(x,w)$, we have 
\begin{equation*}
v(x,w) = \sup_{P\in \cL(x,w)} V_C(x,P) \le  \sup_{({w}'_{\cdot}, {p})\in \cA(x,w)} \delta_1 \int_{\bX} [p_y V_S(y) + (1-p_y) v(y,w'_y)] \Pi(x,dy).
\end{equation*}

"$\ge$". For each $\epsilon > 0$, there exists $(\hat{{w}}_{\cdot}, \hat{p}) \in \cA(x,w)$ such that 
\begin{equation*}
\text{RHS of \eqref{eq:v_bellman} } \le  \delta_1 \int_{\bX} [\hat{p}(y) V_S(y) +(1-\hat{p}(y)) v(y,\hat{w}_y)] \Pi(x,dy) + \epsilon.
\end{equation*}
Since
$$\text{Graph}(\cL(\cdot, \hat w_\cdot))=\{(y,P):\ W_C(y,P)=\hat w_y\}\subset \X\times\cR$$
is Borel measurable,  by measurable selection (see e.g., \cite{doi:10.1137/0315056}) there exists Borel measurable $\mathbf{\hat{P}}: \X \to\cR$ such that $ \mathbf{\hat{P}}(y)\in \cL(y, \hat{w}_y)$ for any $y \in \bX$,  and 
\begin{equation*}
V_C(y, \mathbf{\hat{P}}(y)) \ge v(y,\hat{w}_y) - \epsilon,\quad \mu\text{-a.s.}\ y\in\X.
\end{equation*}
Let ${P} = \hat{p} \oplus \mathbf{\hat{P}} \in \cR$. Then ${P} \in \cL(x,w)$ since $w = \Theta_x(\hat{w}_{\cdot},\hat{p}) = W_C(x, {P})$.
Therefore, 
\begin{equation*}
\begin{split}
\text{RHS of \eqref{eq:v_bellman} } &\le  \delta_1 \int_{\bX} [\hat{p}_y V_S(y) +(1-\hat{p}_y) V_C(y, \mathbf{\hat{P}}(y))] \Pi(x,dy) + 2\epsilon \\ 
&=   V_C(x,{P}) + 2\epsilon \le v(x,w_x) + 2\epsilon.
\end{split}
\end{equation*}
Letting $\epsilon \to 0$, we obtain the desired result.

Next we show that \eqref{eq:v_bellman} admits a unique solution. Suppose that $\tilde{u}, \hat{u}: D \to \bR$ are two bounded and Borel measurable functions that satisfy equation \eqref{eq:v_bellman}. Then for each $x \in \bX$ and each $w \in D_x$, we have 
\begin{equation*}
\left|\hat{u}(x,w) - \tilde{u}(x,w)\right| \le \sup_{({w}'_{\cdot}, p)\in \cA(x,w)} \delta_1 \int_{\bX}(1-p_y)\left|\hat{u}(y,w'_y) - \tilde{u}(y,w'_y)\right| \Pi(x,dy)\le \delta_1 \|\hat{u} - \tilde{u}\|_{\infty}.
\end{equation*}
Therefore, \[
 \|\hat{u} - \tilde{u}\|_{\infty} = \sup_{x\in \bX}\sup_{w \in D_x} \left|\hat{u}(x,w) - \tilde{u}(x,w)\right| \le \delta_1 \|\hat{u} - \tilde{u}\|_{\infty}.
\]
Since $\delta_1<1$, we have $\|\hat{u} - \tilde{u}\|_{\infty} = 0$ and the solution to \eqref{eq:v_bellman} is unique.
\end{proof}

\subsection{Proof in Section \ref{sec:infinity_equilibrium}}

\begin{proof}[Proof of Proposition \ref{prop:ex_inf_horz}]
We prove the proposition by presenting a counterexample in which no randomized equilibrium strategy exists. This example, originally constructed in \cite{MR4397932}, is a variation of a war-of-attrition-type game.

Let $\bX = \{a,b,c\}$ and consider the transition matrix $\Pi = (\pi_{xy})_{x,y \in \bX}$, where all $\pi_{xy} > 0$ except $\pi_{ac} = 0$. Suppose the functions are given by the following equations: 
\[
\begin{array}{llllll}
f_1(a) = 1, &f_1(b) = K, &f_1(c) = 1, &f_2(a) = K, &f_2(b) = 1, &f_2(c) = K^2,\\
g_1(a) = K^2, &g_1(b) = K+1, &g_1(c) = 2, &g_2(a) = K+1, &g_2(b) = 2, &g_2(c) = K^2+1,
\end{array}
\]
where $K>0$ is a constant. We assume $h_i(x) = \frac{1}{2}\left(f_i(x) + g_i(x)\right)$ for $i \in \{1,2\}, x \in \bX$.

We will show that no relaxed equilibrium exists if $K$ is large enough.

Note that the functions satisfy the inequality $f_i(x) < h_i(x) < g_i(x)$ for all $x\in \bX, i \in \{1,2\}$. It implies that for every ${p} \in [0,1]^3$, we have
\begin{equation*}
W_S(x) = g_2(x), V_S(x) = f_1(x), \forall x \in \{a,b,c\}.
\end{equation*}

Recall that $V(x,{p}) = p_x V_S(x) + (1-p_x) V_C(x, {p})$ and $W(x,{p}) = p_x W_S(x) + (1-p_x) W_C(x, {p})$, where $W_C$ and $V_C$ satisfy equations~\eqref{eq:w_markov_bellman} and \eqref{eq:v_markov_bellman}, respectively.

We now derive a set of  necessary  conditions that any equilibrium strategy must satisfy, and subsequently show that no Markov strategy meets these conditions. To that end, suppose ${p}$ is an equilibrium strategy. In what follows, we simplify notation by writing $V(x)$ and $W(x)$ in place of $V(x, {p})$ and $W(x, {p})$, respectively. For each $x\in \{a,b,c\}$, we consider the three scenarios below:

\begin{itemize}
	\item[1.] $p_x = 0$. We have
	\begin{align*}
	W(x) &= W_C(x,{p}) = \max \left\{f_2(x), \delta_2 \sum_{y\in \bX}\pi_{xy}W(y) \right\},\\
	V(x) &= V_C(x,{p}) = g_1(x) \one_{\{W(x)= f_2(x)\}} + \delta_1 \sum_{y\in \bX}\pi_{xy}V(y) \one_{\{W(x)> f_2(x)\}}.
	\end{align*}
	Since $V(x, {p}'\oplus_1 {p}) = V_S(x) = f_1(x)$ when $p'_x = 1$, the equilibrium condition $V(x) \ge V(x, {p}'\oplus_1 {p})$ implies that
	\begin{equation*}
	\begin{cases}
	g_1(x) > f_1(x),   &\text{ if } W(x)= f_2(x),\\
	\delta_1 \sum_{y\in \bX}\pi_{xy}V(y) \ge f_1(x), &\text{ if } W(x)> f_2(x).
	\end{cases}
	\end{equation*}

	\item[2.] $p_x = 1$. We have 
	\begin{align*}
	W(x) = W_S(x) = g_2(x) \quad \text{ and } \quad V(x) =V_S(x) = f_1(x) .
	\end{align*}
	Note that $V(x,{p}'\oplus_1 {p}) = V_C(x,{p})$ when $p'_x = 0$. If $V(x) \ge V(x, {p}'\oplus_1 {p})$, we have 
	\begin{equation*}
	f_1(x) \ge g_1(x) \one_{\{W_C(x, {p})= f_2(x)\}} + \delta_1 \sum_{y\in \bX}\pi_{xy}V(y) \one_{\{W_C(x, {p})> f_2(x)\}}.
	\end{equation*}
	Since $f_1(x) < g_1(x)$, the above inequality implies that 
	\begin{equation*}
	W_C(x, {p}) > f_2(x) \text{ and }   V(x)=f_1(x) \ge \delta_1 \sum_{y\in \bX}\pi_{xy}V(y).
	\end{equation*}

	\item[3.] $p_x \in (0,1)$. We have 
	\begin{align*}
	W(x) = p_x W_S(x) + (1-p_x) W_C(x, {p})  \quad \text{ and } \quad  V(x) = p_x V_S(x) + (1-p_x) V_C(x, {p}).
	\end{align*}
	Since $V(x) \ge V(x, {p}'\oplus_1 {p})$ when $p'_x \in \{0,1\}$, we have 
	\begin{equation*}
	V_S(x) = V_C(x, {p}) \Rightarrow f_1(x) = g_1(x) \one_{\{W_C(x, {p})= f_2(x)\}} + \delta_1 \sum_{y\in \bX}\pi_{xy}V(y) \one_{\{W_C(x, {p})> f_2(x)\}}.
	\end{equation*}
	Again, since $f_1(x) < g_1(x)$, the above equality implies that 
	\begin{equation*}
	W_C(x, {p})> f_2(x) \text{ and }   V(x)=f_1(x) = \delta_1 \sum_{y\in \bX}\pi_{xy}V(y).
	\end{equation*}
\end{itemize}

In all three cases, $V(x) \ge f_1(x)$ for $x\in \{a,b,c\}$. In particular, $V(b)\ge K$. If $p_a \in (0,1]$, the necessary condition $V(a) = f_1(a) = 1$ will contradict $ 1 \ge \delta_1 (\pi_{aa} +  \pi_{ab} V(b)) \ge \delta_1 (\pi_{aa} +  \pi_{ab} K)$ when $K$ is large. Similarly, if $p_c \in (0,1]$, the condition $V(c) = f_1(c) = 1 \ge \delta_1 (\pi_{ca} V(a) +  \pi_{cb} V(b) + \pi_{cc} V(c))$ also leads to a contradiction. It remains to check the following three cases.

\emph{Case 1.} We have $p_a = 0, p_c = 0$, and $p_b = 1$. The follower's value function satisfies
\begin{align*}
W(a) &= \max  \left\{K, \delta_2 (\pi_{aa}W(a) + \pi_{ab} W(b)) \right\},\\
W(b) &= 2,\\
W(c) &= \max  \left\{K^2, \delta_2  (\pi_{ca}W(a) + \pi_{cb} W(b) + \pi_{cc} W(c))\right\},
\end{align*} 
which leads to $V(a) = g_1(a) = K^2, V(c) = g_1(c) = 2$. However, it contradicts $V(b) = f_1(b) = K \ge \delta_1 (\pi_{ba}V(a)+\pi_{bb}V(b)+\pi_{bc}V(c))$ when $K$ is large.

\emph{Case 2.} We have $p_a = 0, p_c = 0$, and $p_b = 0$. The follower's value function satisfies
\begin{align*}
W(a) &= \max  \left\{K, \delta_2 (\pi_{aa}W(a) + \pi_{ab} W(b)) \right\},\\
W(b) &= \max  \left\{2, \delta_2  (\pi_{ba}W(a) + \pi_{bb} W(b) + \pi_{bc} W(c)) \right\},\\
W(c) &= \max  \left\{K^2, \delta_2  (\pi_{ca}W(a) + \pi_{cb} W(b) + \pi_{cc} W(c))\right\},
\end{align*} 
which leads to 
\begin{align*}
V(a) &= \delta_1 (\pi_{aa}V(a)+\pi_{ab}V(b)),\\ 
V(b) &= \delta_1 (\pi_{ba}V(a)+\pi_{bb}V(b)+\pi_{bc}V(c)),\\ 
V(c) & = 2.
\end{align*} The solution for $V(b)$ in the above equation does not satisfy the necessary condition $V(b) \ge f_1(b) = K$ for large $K$.

\emph{Case 3.} We have $p_a = 0, p_c = 0$, and $p_b \in (0,1)$. The follower's value function satisfies
\begin{align*}
W(a) &= \max  \left\{K, \delta_2 (\pi_{aa}W(a) + \pi_{ab} W(b)) \right\},\\
W(b) &= 2 p_b + (1-p_b) \delta_2 (\pi_{ba}W(a)+\pi_{bb}W(b)+\pi_{bc}W(c)),\\
W(c) &= \max  \left\{K^2, \delta_2  (\pi_{ca}W(a) + \pi_{cb} W(b) + \pi_{cc} W(c))\right\}.
\end{align*} 
Since $K^2$ is large, $W(x) \le K^2$ for all $x \in \{a,b,c\}$. We have $W(c) =K^2$ and $V(c) = 2$. Then depending on whether or not the follower stops, we have either
\begin{align*}
V(a) &= g_1(a) = K^2,\\ 
V(b) &= K = \delta_1 (\pi_{ba}V(a)+\pi_{bb}V(b)+\pi_{bc}V(c)),\\ 
V(c) & = 2
\end{align*} 
or
\begin{align*}
V(a) &= \delta_1 (\pi_{aa}V(a)+\pi_{ab}V(b)),\\ 
V(b) &= K = \delta_1 (\pi_{ba}V(a)+\pi_{bb}V(b)+\pi_{bc}V(c)),\\ 
V(c) & = 2
\end{align*} 
holds. It is easy to see that when $K$ is large enough neither equation system holds.
\end{proof}

\subsection{Proofs in Section \ref{sec:regular}}

\begin{proof}[Proof of Lemma \ref{lemma:convergence}]For any $x\in \bX$,
    \begin{align*}
        \left|\bE_x[p^n_{X_1} \phi^n(X_1)]-\bE_x[p_{X_1} \phi(X_1)]\right| \le &\left|\bE_x[p^n_{X_1} \phi^n(X_1)]-\bE_x[p^n_{X_1} \phi(X_1)]\right| + \left|\bE_x[p^n_{X_1} \phi(X_1)]-\bE_x[p_{X_1} \phi(X_1)]\right|\\
        \le &\bE_x[p^n_{X_1} |\phi^n(X_1)-\phi(X_1)|] + \left|\bE_x[p^n_{X_1} \phi(X_1)]-\bE_x[p_{X_1} \phi(X_1)]\right|\\
        \le & \bE_x[|\phi^n(X_1)-\phi(X_1)|] + \left|\int_{\bX} (p^n_y -p_y)\phi(y)\pi(x,y)\mu(dy) \right|.
    \end{align*}
    Since $\phi^n(x) \to \phi(x)$ for a.e.~$x \in \bX$, by Dominant Convergence Theorem we have that the first term satisfies \[
    \lim_{n \to \infty} \bE_x[|\phi^n(X_1)-\phi(X_1)|]  = 0.
    \] 
    Since $\sup_{x\in \bX} |\phi(x)|<\infty$ and  $\pi(x, \cdot) \in L^1(\bX, \mu)$, the weak-* convergence of $\{p^n\}_{n \in \bN}$ implies that the second term also converges to $0$, i.e.,
    \[
    \lim_{n \to \infty} \left|\int_{\bX} (p^n_y -p_y)\phi(y)\pi(x,y)\mu(dy) \right| = 0.
    \]
\end{proof}

\begin{proof}[Proof of Proposition \ref{prop:continuous_wc}]
    Thanks to Lemma~\ref{lemma:r_star}, for each $p \in \cR_0$, the follower's continuation value $W^{\lambda}_C(x,p)$ admits the representation
    \begin{equation}\label{eq:wc_decompse_1}
        W^{\lambda}_C(x, p) = \sup_{q \in \cR_0} \lambda \cH(q_x) +  q_xf_2(x) + (1-q_x)\sum_{t = 1}^{\infty} \delta_2^t H_t(x, p, q).
    \end{equation}
    Here $H_t(x, p, q)$ is defined by \begin{equation*}
        H_t(x,p,q) = \bE_x\left[\left(\prod_{k = 1}^{t-1} (1-q_{X_k})(1-p_{X_k})\right)\left(p_{X_t}W^{\lambda}_S(x) + (1-p_{X_t})\left(q_x f_2(x) +\lambda \cH(q_x)\right)\right)\right],
    \end{equation*}
    with the convention that the empty product $\prod_{k = 1}^{t-1} (1-q_{X_k})(1-p_{X_k})$ equals one when $t = 1$.

    To analyze continuity in $p$, we introduce a finite-horizon truncation of $W^{\lambda}_C$ as follows. For $T \ge 1$, define
    \begin{equation*}
        W^{\lambda,T}_C(x,p) := \sup_{q_t\in \cR_0, t \in [[0, T-1]]}\lambda \cH(q_{0,x}) +  q_{0,x}f_2(x) + (1-q_{0,x})\left( \delta_2^T H_T(x) + \sum_{t = 1}^{T-1} \delta_2^t H_t(x, p, q_t) \right),
    \end{equation*}
where $H_T(x) := \bE_x\left[\left(\prod_{k = 1}^{T-1} (1-q_{X_k})(1-p_{X_k})\right) h_2(X_T)\right]$.

In the finite-horizon problem, the follower's policy $q_t$ is allowed to depend on time. In the infinite-horizon setting, since the follower's problem is stationary when the leader adopts a Markov strategy $p\in \cR_0$, the optimal value \eqref{eq:wc_decompse_1} can equivalently be characterized as
\begin{equation*}
        W^{\lambda}_C(x, p) = \sup_{q_t \in \cR_0, t\in \bN_0} \lambda \cH(q_{0,x})) +  q_{0,x}f_2(x) + (1-q_{0,x})\sum_{t = 1}^{\infty} \delta_2^t H_t(x, p, q_{t ,x}).
    \end{equation*}

Since $f_2,g_2$ and $h_2$ are bounded functions on $\bX$, there exists a constant $M>0$ such that, for any $T \in \bN$, \[
\left|W^{\lambda,T}_C(x,p) - W^{\lambda}_C(x,p) \right| \le \sum_{t = T}^{\infty} \delta_2^t M = \frac{\delta_2^T M}{1-\delta_2}, \quad \forall x \in \bX, p \in \cR_0.
\]
Consequently, to establish the continuity of $W^{\lambda}_C(x,p)$ w.r.t. $p$, it suffices to show that for each $x \in \bX$ and  $T \in \bN$,
\begin{equation}\label{eq:WT_continuous}
    p^n \xrightarrow{w*} p \quad \Longrightarrow \quad W^{\lambda,T}_C(x, p^n) \xrightarrow{ } W^{\lambda,T}_C(x,p).
\end{equation}

We prove \eqref{eq:WT_continuous} by induction. When $T = 1$,\[
W^{\lambda,T}_C(x,p) = \sup_{q_{0,x} \in [0,1]}\lambda \cH(q_{0,x}) +  q_{0,x}f_2(x) + (1-q_{0,x}) \delta_2 \bE_x[h_2(X_1)], 
\]
which is independent of $p$. Thus \eqref{eq:WT_continuous} holds trivially.
Assume that \eqref{eq:WT_continuous} holds for $T = 1, \cdots, l$. Now consider $T = l+1$ for $l\ge 1$.

By classical DPP, $W^{\lambda,l+1}_C(x,p)$ satisfies the following Bellman equation.
\begin{equation*}
    W_C^{\lambda,l+1}(x,p) = \sup_{q_{0,x} \in [0,1]}\lambda \cH(q_{0,x}) +  q_{0,x}f_2(x) + (1-q_{0,x}) \delta_2 \bE_x\left[p_{X_1}W^{\lambda}_S(X_1) + (1-p_{X_1})W^{\lambda,l}_C(X_1,p)\right].
\end{equation*}
The above supremum is attained at 
\begin{equation*} 
q^*_{0,x} :=q^*_{0,x} (p):= \frac{1}{1+ \exp \left(\frac{\delta_2 \bE_x[p_{X_1}W^{\lambda}_S(X_1) + (1-p_{X_1})W^{\lambda,l}_C(X_1,p)]- f_2(x)}{\lambda}\right)}.
\end{equation*} 
Then we have 
\begin{equation*}
W^{\lambda,l+1}_C(x, {p})  = f_2(x) + \lambda \log\left(1+\exp\left(\frac{\delta_2 \bE_x[p_{X_1}W^{\lambda}_S(X_1) + (1-p_{X_1})W^{\lambda,l}_C(X_1,p)]- f_2(x)}{\lambda}\right)\right).
\end{equation*}
Thanks to Lemma~\ref{lemma:convergence} and the induction hypothesis, $\bE_x[p_{X_1}W^{\lambda}_S(X_1)]$ and $ \bE_x[(1-p_{X_1})W^{\lambda,l}_C(X_1,p)]$ are continuous w.r.t.~$p$. Consequently, $W^{\lambda,l+1}_C(x, {p})$ is continuous w.r.t.~$p$.
\end{proof}

\begin{proof}[Proof of Proposition \ref{prop:equilibrium_characterization}]
    Thanks to \eqref{eq:v_lambda_decompose}, for each $x \in \bX$, the condition \[
    V^{\lambda}(x,p) \ge V^{\lambda}(x, p' \oplus_1 p), \,\, \forall p' \in \cR_0
    \]  
    holds,
    if and only if  \begin{equation*}
        p_x V_S^{\lambda}(x) + (1-p_x)V^{\lambda}_C(x,p) = \max_{p'_x \in [0,1]} p'_x V_S^{\lambda}(x) + (1-p'_x)V^{\lambda}_C(x,p), 
    \end{equation*}
    which is equivalent to 
    \[
    (1- p_x )\one_{\{V_S^{\lambda}(x) > V_C^{\lambda}(x,p\}} = p_x \one_{\{V_S^{\lambda}(x) < V_C^{\lambda}(x,p\}}  = 0.
    \]
     
    If $p\in \Psi(p)$, then
    \[
    (1- p_x )\one_{\{V_S^{\lambda}(x) > V_C^{\lambda}(x,p\}} = p_x \one_{\{V_S^{\lambda}(x) < V_C^{\lambda}(x,p\}} = 0, \quad \mu\text{-a.s.},
    \]
    since $(1- p_x )\one_{\{V_S^{\lambda}(x) > V_C^{\lambda}(x,p\}} \ge 0$ and $ p_x \one_{\{V_S^{\lambda}(x) < V_C^{\lambda}(x,p\}} \ge 0$.
    \end{proof}

\begin{proof}[Proof of Proposition \ref{prop:continuous_vc}]
Given the follower's best response \eqref{eq:q_star}, the leader's continuation value $V^{\lambda}_C$ admits the following representation. For each $x \in \bX, p \in \cR_0$,
\begin{equation*}
    V^{\lambda}_C(x,p) = q^*_x(p)g_1(x) + (1-q^*_x(p))\sum_{t = 1}^{\infty} \delta_1^t I_t(x,p),
\end{equation*}
where $I_t$ is defined by
\begin{equation*}
    I_t(x,p) := \bE_x\left[\left(\prod_{k = 1}^{t-1}\left(1-q^*_{X_k}(p)\right)\left(1-p_{X_k}\right)\right) \left(p_{X_t} V_S^{\lambda}(X_t) + (1-p_{X_t})q^*_{X_t}(p) g_1(X_t)\right)\right].
\end{equation*}
To establish continuity of $V^{\lambda}_C$ w.r.t.~$p$, it suffices to show that for each $x \in \bX$ and  $t \in \bN$,
\begin{equation}\label{eq:It_continuous}
    p^n \xrightarrow{w*} p \quad \Longrightarrow \quad I_t(x, p^n) \xrightarrow{ } I_t(x,p).
\end{equation}

We prove \eqref{eq:It_continuous} by induction. When $t = 1$,\[
I_t(x,p) = \bE_x[p_{X_1}V^{\lambda}_S(X_1) + (1-p_{X_1})q^*_{X_1}(p)g_1(X_1)].
\]
By Lemma~\ref{lemma:convergence} and Corollary~\ref{cor:continuous_qstar}, the map $p \mapsto I_1(x, p)$ is continuous.
Assume that \eqref{eq:It_continuous} holds for $t = 1, \cdots, l$. Now consider $t = l+1$ for $l\ge 1$.

Conditioning on the $\sigma$-algebra generated by $X_1$, we obtain
\begin{equation*}
I_{l+1}(x,p) = \bE_x\left[\left(1-q^*_{X_1}(p)\right)(1-p_{X_1})I_{l}(X_1,p)\right].
\end{equation*}
Due to the induction hypothesis and Corollary~\ref{cor:continuous_qstar}, \[
p^n\xrightarrow{w*}p \Longrightarrow I_{l}(x,p^n)\left(1-q^*_{X_1}(p^n)\right) \to I_l(x, p)\left(1-q^*_{X_1}(p)\right), \, \forall x \in \bX.
\]
Applying Lemma~\ref{lemma:convergence} again completes the proof. 
\end{proof}

\begin{proof}[Proof of Theorem \ref{thm}]
By Proposition~\ref{prop:equilibrium_characterization} and Remark~\ref{rmk:mu_null}, any fixed-point $p \in \Psi(p)$ can be modified on a $\mu$-null set to yield a regular randomized equilibrium in the sense of Definition~\ref{def:infinite_equilibrium}. Henceforth, it suffices to prove the existence of a fixed point of $\Psi$.
Since $\cR_0$ is a non-empty, compact and convex subset of $L^{\infty}(\bX,\mu)$, we can apply Kakutani's theorem to prove the existence of fixed point.

For each fixed $p \in \cR_0$, $V^{\lambda}_S(\cdot)$ and $V^{\lambda}_C(\cdot,p)$ are measurable  and bounded functions on $\bX$, set $\tilde p_{\cdot} = \one_{\{V_S^{\lambda}(\cdot) \ge V_C^{\lambda}(\cdot,p) \}} \in \cR_0$. Then $\tilde p \in \Psi(p)$ and hence $\Psi(p)$ is non-empty. It is straightforward that if $\tilde p, \hat p \in \Psi(p)$, then \[\alpha \tilde p + (1-\alpha)\hat p \in \Psi(p), \quad \alpha \in [0,1].\] which means that $\Psi(p)$ is convex for all $p \in \cR_0$.

Next we show that for any $p^n, \tilde{p}^n, p, \tilde{p}  \in  \cR_0$ with $p^n \xrightarrow{w*} p$ and $\tilde{p}^n \xrightarrow{w*} \tilde{p}$, if $\tilde{p}^n \in \Psi(p^n)$, then $\tilde{p} \in \Psi(p)$. Equivalently, we need to show that \begin{equation} \label{eq:condition1}
    \int_{\bX} \tilde p_x \one_{\{V_S^{\lambda}(x) < V^{\lambda}_C(x,p)\}} \mu(dx) = 0,
\end{equation}
and 
\begin{equation} \label{eq:condition2}
    \int_{\bX} (1-\tilde p_x) \one_{\{V_S^{\lambda}(x) > V^{\lambda}_C(x,p)\}} \mu(dx) = 0.
\end{equation}

For the fixed $p \in \cR_0$, define\[
A := A_p := \left\{y: V_S^{\lambda}(y) < V^{\lambda}_C(y,p) \right\},
\]
which is a measurable subset of $\bX$.

Since $\lim_{n \to \infty} V^{\lambda}_C(x, p^n) = V^{\lambda}_C(x, p)$, if $V_S^{\lambda}(x) < V^{\lambda}_C(x,p)$, then for large $n$, $V_S^{\lambda}(x) < V^{\lambda}_C(x,p^n)$. Consequently, for $ x \in A$,
\[
\lim_{n \to \infty}\tilde p^n_x\left(\one_{\{V_S^{\lambda}(x) < V^{\lambda}_C(x,p^n)\}} - \one_{\{V_S^{\lambda}(x) < V^{\lambda}_C(x,p)\}} \right)=  0, 
\]
By Dominant Convergence Theorem,
\[
\lim_{n\to \infty} \int_A \tilde p^n_x\left(\one_{\{V_S^{\lambda}(x) < V^{\lambda}_C(x,p^n)\}} - \one_{\{V_S^{\lambda}(x) < V^{\lambda}_C(x,p)\}} \right) \mu(dx) = 0.
\]
For $ x \notin A$, by definition,
\[
p^n_x\left(\one_{\{V_S^{\lambda}(x) < V^{\lambda}_C(x,p^n)\}} - \one_{\{V_S^{\lambda}(x) < V^{\lambda}_C(x,p)\}}\right) \ge 0, \quad \forall n\in \bN.
\]
By Fatou's Lemma,
\begin{align*}
&\liminf_{n \to \infty} \int_{A^c}\tilde p^n_x\left(\one_{\{V_S^{\lambda}(x) < V^{\lambda}_C(x,p^n)\}} - \one_{\{V_S^{\lambda}(x) < V^{\lambda}_C(x,p)\}}\right)  \mu(dx)\\
&\ge \int_{A^c} \liminf_{n \to \infty}\tilde p^n_x\left(\one_{\{V_S^{\lambda}(x) < V^{\lambda}_C(x,p^n)\}} - \one_{\{V_S^{\lambda}(x) < V^{\lambda}_C(x,p)\}}\right)  \mu(dx) 
\ge 0.
\end{align*}

Therefore, we obtain\begin{equation}\label{eq:liminf}
\liminf_{n \to \infty} \int_{\bX}\tilde p^n_x\left(\one_{\{V_S^{\lambda}(x) < V^{\lambda}_C(x,p^n)\}} - \one_{\{V_S^{\lambda}(x) < V^{\lambda}_C(x,p)\}}\right)  \mu(dx) \ge 0
\end{equation}
Since $\tilde p^n \in \Psi(p^n)$,
\[
\int_{\bX} \tilde p^n_x \one_{\{V_S^{\lambda}(x) < V^{\lambda}_C(x,p^n)\}} \mu(dx) = 0, \quad \forall  n \in \bN.
\]
Using $\tilde p^n \xrightarrow{w*} p$, the above inequality~\eqref{eq:liminf} leads to \[
0 \leq \int_{\bX} \tilde p_x \one_{\{V_S^{\lambda}(x) < V^{\lambda}_C(x,p)\}} \mu(dx)=\limsup_{n \to \infty}\int_{\bX} \tilde p^n_x \one_{\{V_S^{\lambda}(x) < V^{\lambda}_C(x,p)\}} \mu(dx) \le 0,
\]
which proves condition~\eqref{eq:condition1} must hold. By similar arguments, condition~\eqref{eq:condition2} also holds.
\end{proof}

\begin{proof}[Proof of Proposition \ref{prop:approximation}]

For any fixed $\epsilon >0$, choose $\lambda <\frac{(1-\delta_2)\epsilon}{\log 2}$. By Theorem~\ref{thm}, there exists a regular equilibrium $p \in \cR_0$ for the entropy regularized Stackelberg stopping game \eqref{eq:game_regular_equilibrium} with parameter $\lambda$. 
By Definition \ref{def:regular_equilibrium}, the strategy pair $(p, (r,q))$ with $(r,q) := (r^{*,\lambda}, q^{*\lambda}(p))$ as defined in \eqref{eq:r_star} and \eqref{eq:q_star},  satisfies that for all $ x \in \bX$,
\begin{align*}
 &J^{\lambda}_2(x, p' \oplus_1 p, (r, q)) \ge J^{\lambda}_2(x, p' \oplus_1 p, (r',q')), \quad \forall (r',q') \in \cR_0^2, \quad \forall p' \in \cR_0;\\
  &J_1(x, {p}, (r, q)) \ge J_1(x, p' \oplus_1 p, (r, q)), \quad \forall p' \in \cR_0.
\end{align*}

Therefore, condition \eqref{eq:j1condition_epsilon} is satisfied. It remains to show that condition \eqref{eq:j2condition_epsilon}
also holds.

Since the difference between $J_2$ and $J_2^{\lambda}$, the regularization term, is bounded by $\lambda\sum_{t  = 0}^{\infty}\delta_2^t |\cH(q_t) -\cH(q'_t)| \le \epsilon$, we obtain
\begin{equation*}
J_2(x, p' \oplus_1 p, (r, q)) \ge J_2(x, p' \oplus_1 p, (r', q')) - \epsilon, \quad \forall x \in \bX, \forall (r',q') \in \cR_0^2,  \forall p' \in \cR_0,
\end{equation*}
which implies that $(p,(r,q))$ is an $\epsilon$-equilibrium.
\end{proof}

\subsection{Proof of Section \ref{sec:special_case}}

\begin{proof}[Proof of Proposition \ref{prop:feedback}]
Thanks to \eqref{eq:r_star_new}, \eqref{eq:vs_limit}, and Proposition~\ref{prop:feedback_char}, it suffices to show that  the sequence $\left(p^n,\; q^{*,\lambda_n}(p^n)\right)$ admits a limit point $(p^*, q^*) \in [0,1]^N \times [0,1]^N$ that satisfies  \eqref{eq:j2condition} and \eqref{eq:j1condition}.

Consider a sequence $\{\lambda_n\}_{n \in \bN}$ that satisfies $\lim_{n \to \infty} \lambda_n = 0$. By Theorem~\ref{thm}
 {\footnote{Assumption~\ref{asp:reference_prob} holds when the state space is finite. We can simply choose the reference distribution to be $\mu(x) = 1/N, \forall x \in \bX$ and the kernel to be $\pi(x,y) = N\bP(X_{t+1}= y | X_t = x ), \forall x,y \in \bX$.}}
 ,
 for each $\lambda_n$, there exists a regular randomized equilibrium $p^n:=p^{\lambda_n} \in [0,1]^N$ for the entropy-regularized Stackelberg game with parameter $\lambda_n$.

  Let us denote the corresponding best response function of the follower by $q^n := q^{*, \lambda_n}(p^n)$. Then for each $x \in \bX$,
\begin{equation} \label{eq:qn}
q^n_x = \frac{1}{1+ \exp\left( \frac{\delta_2 \bE_x \left[ W^{\lambda_n}(X_1, p^n) \right] - f_2(x)}{\lambda_n} \right)}, \end{equation}
and 
\begin{equation} \label{eq:pn}
p^n_x = 
\begin{cases}
1, & \text{ if } V_S^{\lambda^n}(x) > V_C^{\lambda^n}(x, p^n);\\
0, & \text{ if } V_S^{\lambda^n}(x) < V_C^{\lambda^n}(x, p^n).
\end{cases}
\end{equation}
Recall that $V_S^{\lambda_n}(x)=r^n_xh_1(x) + (1-r^n_x) f_1(x)$, where $r^n = r^{*,\lambda_n}$ is given by \eqref{eq:r_star}, and \[
V^{\lambda_n}_C(x,p^n) = q^n_x g_1(x)+(1-q^n_x) \delta_1 J_1(X_1, {p^n}, (r^{*, \lambda_n},q^n)).
\]  
Since $\bX$ is finite, we can define $p^* := \lim_{n \to \infty} p^n$, up to a proper subsequence. 
Next we select a limit point $q^*$ of $\{q^n\}_{n \in \bN}$ such that the pair $(p^*,q^*)$ satisfies conditions \eqref{eq:j2condition}
and \eqref{eq:j1condition}.

Recall that for any $p \in \cR_0$,
\begin{align*}
&W^{\lambda}(x,p) = \sup_{q,r \in [0,1]^N} J_2(x, p, (r,q)) + \lambda \bE_x\left[ \sum_{t = 0}^{\tau^p \wedge \tau^{q,r}} \delta_2^t \cH(q_{X_t}  \one_{\{\tau^p>t\}} + r_{X_t}  \one_{\{\tau^p= t\}} ) \right];\\
&W(x,p) = \sup_{q,r \in [0,1]^N} J_2(x, p, (r,q)).
\end{align*}
Then\[
\left| W^{\lambda}(x,p)-W(x,p) \right| \le \sup_{q,r \in [0,1]^N} \lambda \bE_x\left[ \sum_{t = 0}^{\tau^p \wedge \tau^{q,r}} \delta_2^t \cH(q_{X_t}  \one_{\{\tau^p>t\}} + r_{X_t}  \one_{\{\tau^p= t\}} ) \right] \le \lambda \frac{2}{1-\delta_2}.
\]
Consequently, $\lim_{n \to \infty}W^{\lambda_n}(x,p) = W(x,p)$ uniformly in $p$.

Moreover, one can easily show that the Bellman operator associated with $W(x, p)$ is a contraction under the supremum norm. Then a Banach fixed-point argument implies that $W(x, p)$ is continuous w.r.t.~ $p$, where the set $[0,1]^N$ is equipped with the supremum norm $\|p\|_{\infty} := \sup_{x\in \bX} |p_x|$.

Since \[
\left|  W^{\lambda_n}(x, p^n)  - W(x, p^*)  \right| \le \left| W^{\lambda_n}(x, p^n) - W(x, p^n)  \right| + \left| W(x, p^n) - W(x, p^*)  \right|,
\]
using the uniform convergence of $W^{\lambda_n}$ and the continuity of $W(x,\cdot)$, we obtain \[
\lim_{n \to \infty} W^{\lambda_n}(x, p^n)  = W(x, p^*).
\]
Consequently, thanks to  \eqref{eq:qn}, if $\delta_2 \bE_x \left[ W(X_1, p^*) \right]  > f_2(x)$,  $\lim_{n \to \infty} q^n_x = 0$, and if $\delta_2 \bE_x \left[ W(X_1, p^*) \right]  < f_2(x)$,  $\lim_{n \to \infty} q^n_x = 1$.
If  $\delta_2 \bE_x \left[ W(X_1, p^*) \right] = f_2(x)$, by choosing a convergent subsequence of $\{q^n_x\}$, we have $\lim_{n \to \infty} q^{n}_x \in [0,1]$. Hence the limit point \[
\quad q^*_x =  \lim_{n \to \infty} q^n_x, \quad \forall x \in \bX,
\]
satisfies that 
\begin{equation*}
q^*_x  f_2(x)+ (1-q^*_x )\delta_2 \bE_x \left[ W(X_1, p^*) \right]  = \max_{q_x\in [0,1]} q_x f_2(x)+ (1-q_x)\delta_2 \bE_x \left[ W(X_1, p^*) \right]  , \quad \forall x \in \bX,
\end{equation*}
which implies that  the pair $(p^*,q^*)$ satisfied condition~\eqref{eq:j2condition}. It remains to verify that the pair $(p^*,q^*)$ satisfies condition~\eqref{eq:j1condition}.

By definition, $J_1(x, \cdot, \cdot)$ is continuous w.r.t.~$p, q$ and $r$. Therefore,\[
J_1(x, {p^n}, (r^{*, \lambda_n}, q^n) )\to  J_1(x, {p^*}, (r^*,q^*)),
\]
and 
\[
\lim_{n \to \infty} V_C^{\lambda_n}(x, p^n) =\lim_{n \to \infty}  q^n_x g_1(x)+(1-q^n_x) \delta_1 J_1(x, {p^n}, (r^{*, \lambda_n}, q^n)) = q^*_x g_1(x)+(1-q^*_x) \delta_1 J_1(x, {p^*}, (r^*, q^*))  = V_C(x, p^*). 
\]
 Using $\lim_{n \to \infty} V_C^{\lambda_n}(x, p^n) = V_C(x, p^*), \lim_{n \to \infty} V_S^{\lambda_n}(x) = V_S(x)$ (see \eqref{eq:vs_limit}) and \eqref{eq:pn}, we conclude that 
\begin{equation*}
p^*_x  = \lim_{n \to \infty}p^n_x = 
\begin{cases}
1, & \text{ if } V_S(x) > V_C(x, p^*);\\
0, & \text{ if } V_S(x) < V_C(x, p^*),
\end{cases}
\end{equation*}
which implies that the pair $(p^*,q^*)$ satisfies condition~\eqref{eq:j1condition}.
\end{proof}

\begin{proof}[Proof of Theorem \ref{thm:strict_condition}]
For each $p \in \bX$, define\[
q^*_x(p) = \one_{f_2(x)\ge \delta_2\bE_x[p_{X_1}W_S(X_1) + (1-p_{X_1})W_C(X_1,p)]}.
\]
Then we have that for a.e. $x \in \bX$, $q^*_x \in \{0,1\}$. Specifically, \[
q^*_x(p) = \begin{cases}
    1, & \text{ if } f_2(x) > \delta_2\bE_x[p_{X_1}W_S(X_1) + (1-p_{X_1})W_C(X_1,p)]\\
    0, & \text{ if } f_2(x) < \delta_2\bE_x[p_{X_1}W_S(X_1) + (1-p_{X_1})W_C(X_1,p)]
\end{cases}\,\,, \quad \text{a.e. } x \in \bX.
\]
Following the same argument, the continuity result in Proposition \ref{prop:continuous_wc} also holds for $W_C$. Consequently, \[
p^n \xrightarrow{w*} p \quad \Rightarrow  \quad q^{*}_x(p^n) \to q^{*,\lambda}_x(p), \quad \text{a.e. } x \in \bX. 
\]
Again following the same fixed-point arguments in the proof of Theorem \ref{thm}, we obtain the existence of the exact equilibrium.
\end{proof}

\begin{proof}[Proof of Proposition \ref{cor:example}]
Thanks to Theorem \ref{thm:strict_condition}, it suffices to show that $\mu(A(p)) = 0$ for all $p \in \cR_0$.
Denote \[
H^p(y):= p_yW_S(y)+ (1-p_y)W_C(y,p), \quad \forall p\in \cR_0, y \in \bX.
\]
Then we have that \[
W_C(x,p) = \max\{x, \delta_2 \bE_x[H^p(X_1)]\},
\]
where
\begin{equation}\label{ft}
\bE_x[H^p(X_1)] =  \int_{0}^{1} (1+\eta (2y-1)(2x-1))[p_y+ (1-p_y)W_C(y,p)]dy
\end{equation}
is a linear function of $x$.

Since $W_C(x,p) \ge x$, we have $p_x+(1-p_x)W_C(x,p) \ge x$. Then by direct computation,\[
\begin{cases}
\delta_2 \bE_x[H^p(X_1)] > x, & \text{ when }\quad x = 0;\\
\delta_2 \bE_x[H^p(X_1)] < x,  & \text{ when }\quad x = 1.
\end{cases}
\]
Thus there exists a unique point $b_p \in (0,1)$ such that $x = \delta_2 \bE_x[H^p(X_1)]$ and we obtain
$$\mu(A(p))=\mu(\{x: x = \delta_2\bE_x[H^p(X_1)]\}) = 0.$$
\end{proof}

\bibliographystyle{plain}
\bibliography{ref}

\end{document}